\documentclass{amsart}
\usepackage{amssymb}
\usepackage{amsmath}
\usepackage{amsfonts}

\newtheorem{theorem}[equation]{Theorem}
\newtheorem{lemma}[equation]{Lemma}
\newtheorem{proposition}[equation]{Proposition}

\newtheorem{example}[equation]{Example}
\newtheorem{corollary}[equation]{Corollary}
\newtheorem{remark}[equation]{Remark}
\newtheorem{definition}[equation]{Definition}

\numberwithin{equation}{section}

\begin{document}
\title{Local One-Sided Porosity And Pretangent Spaces}

\author{M. Alt{\i}nok, O. Dovgoshey, M. K\"{u}\c{c}\"{u}kaslan}

\address[M. Alt{\i}nok]{Department of Mathematics, Faculty of Sciences and Arts, Mersin University, Mersin, 33343 TURKEY}
\email{mayaaltinok@mersin.edu.tr}
\address[O. Dovgoshey]{Division of Applied Problems in Contemporary Analysis, Institute of Mathematics of NASU, Tereschenkivska st. 3, Kiev-4, 01601 UKRAINE}
\email{aleksdov@mail.ru}
\address[M. K\"{u}\c{c}\"{u}kaslan]{Department of Mathematics, Faculty of Sciences and Arts, Mersin University, Mersin, 33343 TURKEY}
\email{mkucukaslan@mersin.edu.tr, mkkaslan@gmail.com}

\date{}

\begin{abstract}
For subsets of $\mathbb{R}^+$ we consider the local right upper porosity and the local right lower porosity as elements of a cluster set of all porosity numbers. The use of a scaling function $\mu:\mathbb{N} \to \mathbb{R}^+$ provides an extension of the concept of porosity numbers on subsets of $\mathbb{N}$. The main results describe interconnections between porosity numbers of a set, features of the scaling funtions and the geometry of so-called pretangent spaces to this set.
\end{abstract}

\maketitle

\section{Introduction}

The porosity appeared in the papers of Denjoy \cite{denjoy1}, \cite{denjoy2} and Khintchine \cite{khintchine} and, independently, Dolzenko \cite{dolzenko}. The porosity has found interesting applications in the theory of free boundaries \cite{karp}, generalized subharmonic functions \cite{dovgoshey}, complex dynamics \cite{przytycki}, quasisymmetric maps \cite{vaisala}, infinitesimel geometry \cite{bilet} and other areas of mathematics.

\begin{definition}\label{d1}
\cite{thomson} Let $E \subseteq \mathbb{R}^+$ where $\mathbb{R}^+=[0, \infty)$. The right upper porosity of $E$ at $0$ is the number
\begin{equation}\label{eq1}
\overline{p} (E)= \limsup_{h \to 0^+} \frac{\lambda (E,h)}{h}
\end{equation}
where $\lambda (E,h)$ is the length of the largest open subinterval of $(0,h)$ that contains no point of $E$.
\end{definition}
The notion of right lower porosity of $E$ at $0$ is defined similarly.

\begin{definition}\label{d2}
Let $E \subseteq \mathbb{R}^+$. The right lower porosity of $E$ at $0$ is the number
\begin{equation}\label{eq2}
\underline{p}(E)= \liminf_{h \to 0^+} \frac{\lambda(E,h)}{h}.
\end{equation}
where $\lambda(E,h)$ is the same as in Definition \ref{d1}.
\end{definition}

We will use the following terminology. A set $E \subseteq \mathbb{R}^+$ is:

Porous at $0$ if $\overline{p}(E)> 0$;

Strongly porous at $0$ if $\overline{p}(E)=1$;

Nonporous at $0$ if $\overline{p}(E)= 0$.

It should be noted that the standard definitions of porous, strongly porous and nonporous sets use the bilateral porosity at a point instead of the right porosity at a point (see, for example, \cite{thomson}) but the present paper deals only with the right porosity at $0$ and its analogues.

One of the main directions of the successful use of the local porosity machinery is the investigations of cluster sets. See, for example, \cite{dolzenko}, \cite{yanagihara}, \cite{yoshida1}, \cite{yoshida2} and \cite{zajicek}. Thus the deep relationship between the local porosity and the cluster sets is obvious for today. The first point of our paper is an inqlusion of the right upper porosity and the right lower porosity in a cluster set of all porosity numbers.

Let $p$ be a real number and let $E \subseteq \mathbb{R}^+$. We say that $p$ is a porosity number (of $E$ at $0$) if there is a sequence $(h_k)_{k \in \mathbb{N}}$ such that
\begin{equation*}
\lim_{k \to \infty} h_k = 0,~~ h_k >0 ~\text{for every}~k \in \mathbb{N},
\end{equation*}
and
\begin{equation}\label{eq3}
p= \lim_{k \to \infty} \frac{\lambda(E, h_k)}{h_k}.
\end{equation}
The set \textbf{P}(E) of all porosity numbers of $E$ simply is the cluster set of the function
\begin{equation}\label{eq4}
\Phi_E(h)= \frac{\lambda(E,0,h)}{h}
\end{equation}
at the point $0$. It is clear that

\begin{equation*}
 \overline{p}(E)=\max_{p \in \textbf{P}(E)}p~\text{and}~ \underline{p}(E)= \min_{p \in \textbf{P}(E)}p.
\end{equation*}
The standard interpretation of the local porosity at a point as a size of holes near this point can be formalized if we use the methods of the infinitesimal geometry and it is the second point of our paper. For this purpose we employ the so-called pretangent metric spaces recently introduced in \cite{dovgoshey1}, \cite{dovgoshey2} for arbitrary metric spaces. We recall the necessary definitions and results related to pretangent metric spaces in the second section of the paper.

The third section begins with the introduction of porosity numbers at infinity for subsets of $\mathbb{N}$. We do this with the help of a scaling function $\mu:\mathbb{N} \to \mathbb{R}^+$ and a large part of our results is a description of interrelations between the properties of scaling functions and porosity properties of subsets of $\mathbb{N}$. The main motivation here is to obtain a porosity machinery for description of asymptotic expansions, methods of summation, polynomial and rational approximations and other important mathematical constructions on $\mathbb{N}$ anyhow connected with cluster sets. The description of porosity properties of the set of primes also seems to be interesting.

The main result of Section 2 is Theorem \ref{th5} giving an infinitesimal characterization of porosity numbers. In Section 3 it is shown that for every $E \subseteq \mathbb{N}$ and every scaling function $\mu$ the upper porosity at coincedes with the upper porosity at $0$ of the set $\mu(E)$ (see Theorem \ref{th2}). Another result of Section 3 is a simple geometric description of nonporous sets (see Theorem \ref{teo211}).

The first result of Section 4 is Theorem \ref{the1} which describes the porosity numbers at infinity for subsets of $\mathbb{N}$. In Theorem \ref{th1} we describe some condition under which the sets of porosity numbers at infinity coincide for two scaling functions. Theorem  \ref{th4} gives a construction of a ``recursively enumerable'' subset of $\mathbb{R}^+$ having the same set of porosity numbers as a given set $E \subseteq \mathbb{R}^+$.

The main object of study in Section 5 is the subsets of $\mathbb{N}$ having the unitary lower porosity at infinity. We give the structural and infinitesimal characterizations of such sets (see Theorem \ref{p3} and Theorem \ref{teo540} respectively). Moreover in Corollary \ref{p2} we prove the exact inequality $\underline{p}(E) \leq \frac{1}{2}$ for every $E \subseteq \mathbb{R}^+$ with $0 \notin acE$. This inequality probably is not new but we do not have any references here.

The proof of almost all main results are simple and self-contained. The exception is Theorem \ref{teo540} with a proof using some additional results related to pretangent spaces.

\section{Porosity Numbers And Pretangent Spaces}

In this section we give a geometrical interpretation of porosity numbers in the language of pretangent spaces.

Let us recall the construction of pretangent spaces to $E$ in the particular case when $E \subseteq \mathbb{R}^+$. Let $\tilde{r}= (r_n)_{n \in \mathbb{N}}$ be a sequence of positive real numbers tending to zero. In what follows $\tilde{r}$ will be called a normalizing sequence. Let us denote by $\tilde{E}$ the set of all sequences $(x_n)_{n \in \mathbb{N}}$ of points from $E$ with $\lim_{n \to \infty}x_n=0$.

\begin{definition}\label{d5}
Two sequences $\tilde{x}= (x_n)_{n \in \mathbb{N}} \in \tilde{E}$ and $\tilde{y}= (y_n)_{n \in \mathbb{N}} \in \tilde{E}$ are mutually stable w.r.t. $\tilde{r}$ if there s a finite limit
\begin{equation}\label{eq27}
\lim_{n \to \infty} \frac{|x_n-y_n|}{r_n}:= |\tilde{x}-\tilde{y}|_{\tilde{r}}.
\end{equation}
\end{definition}
We shall say that a family $\tilde{F} \subseteq \tilde{E}$ is self-stable (w.r.t. $\tilde{r}$) if every two $\tilde{x},\tilde{y} \in \tilde{F}$ are mutually stable. A family $\tilde{F} \subseteq \tilde{X}$ is maximal self-stable is $\tilde{F}$ is self-stable and for an arbitrary $\tilde{z} \in \tilde{E}$ either $\tilde{z} \in \tilde{F}$ or there is $\tilde{x} \in \tilde{F}$ such that $\tilde{x}$ and $\tilde{z}$ are not mutually stable.

\begin{proposition}\label{p8}
(\cite{dovgoshey1}, \cite{dovgoshey2}) Let $E \subseteq \mathbb{R}^+$ be a pointed set with the marked point $0 \in E$. Then for every normalizing sequence $\tilde{r}= (r_n)_{n \in \mathbb{N}}$ there exists a maximal self-stable family $\tilde{E}_{0, \tilde{r}}$ such that
\begin{equation*}
\tilde{0}:=(0,...,0,0,...) \in \tilde{E}_{0, \tilde{r}}.
\end{equation*}
\end{proposition}

Consider a function $|.,.|_{\tilde{r}}$ on $\tilde{E}_{0, \tilde{r}} \times \tilde{E}_{0, \tilde{r}}$ where $|\tilde{x},\tilde{y}|_{\tilde{r}}= |\tilde{x}- \tilde{y}|_{\tilde{r}}$ is defined by \eqref{eq27}. Obviously, $|.,.|_{\tilde{r}}$ is symmetric, nonnegative and satisfies the inequality
\begin{equation*}
|\tilde{x}- \tilde{y}|_{\tilde{r}} \leq |\tilde{x}- \tilde{z}|_{\tilde{r}} + |\tilde{z}- \tilde{y}|_{\tilde{r}}
\end{equation*}
for all $\tilde{x},\tilde{y},\tilde{z} \in \tilde{E}_{0, \tilde{r}}$. Hence $(\tilde{E}_{0, \tilde{r}}, |.,.|_{\tilde{r}})$ is a pseudometric space.

\begin{definition}\label{d6}
A pretangent space to $E \subseteq \mathbb{R}^+$ (at the point $0 \in E$ w.r.t. $\tilde{r}$) is the metric identification of a pseudometric space $(\tilde{E}_{0, \tilde{r}}, |.,.|_{\tilde{r}})$.
\end{definition}

Since the notion of pretangent space is important for the paper, we shall describe the metric identification construction (see, for example, \cite{ke}). Define a binary relation $\thicksim$ on $\tilde{E}_{0, \tilde{r}}$ by $\tilde{x} \thicksim \tilde{y}$ if and only if $|\tilde{x}- \tilde{y}|_{\tilde{r}}= 0$. Then $\thicksim$ is an equivalence relation. Let us denote by $\Omega^{E}_{0, \tilde{r}}$ the set of equivalence classes in $\tilde{E}_{0, \tilde{r}}$ under the equivalence relation $\thicksim$. If for arbitrary $\alpha, \beta \in \Omega^{E}_{0, \tilde{r}}$ and $\tilde{x} \in \alpha, \tilde{y} \in \beta$, we set
\begin{equation}\label{eq28}
\rho (\alpha, \beta):= |\tilde{x}- \tilde{y}|_{\tilde{r}},
\end{equation}
then $\rho$ is a well-defined metric on $\Omega^{E}_{0, \tilde{r}}$. The metric space $(\Omega^{E}_{0, \tilde{r}} , \rho)$ is, by definition, the metric identification of $(\tilde{E}_{0, \tilde{r}}, |.,.|_{\tilde{r}})$.

\begin{proposition}\label{p9}
Let $E \subseteq \mathbb{R}^+$ be a pointed set with a marked point $0 \in E$ and let $\tilde{E}_{0, \tilde{r}}$ be a maximal self-stable family w.r.t. a normalizing sequence $\tilde{r}=(r_n)_{n \in \mathbb{N}}$. Then, for every pair $\tilde{x}, \tilde{y} \in \tilde{E}_{0, \tilde{r}}$, the statement $\tilde{x} \thicksim \tilde{y}$ holds if and only if
\begin{equation*}
\lim_{n \to \infty} \frac{x_n}{r_n} = \lim_{n \to \infty} \frac{y_n}{r_n}.
\end{equation*}
\end{proposition}
\begin{proof}
Since $\tilde{x}$ and $\tilde{0}$ are mutually stable, we have
\begin{equation*}
|\tilde{x}- \tilde{0}|_{\tilde{r}}= \lim_{n \to \infty} \frac{x_n - 0}{r_n} = \lim_{n \to \infty} \frac{x_n}{r_n} < \infty.
\end{equation*}
Similarly there exists the finite limit $\lim_{n \to \infty} \frac{y_n}{r_n}$. From the definition of $\thicksim$ we have $\tilde{x} \thicksim \tilde{y}$ if and only if $\lim_{n \to \infty} \frac{|x_n- y_n|}{r_n}=0$. Since
\begin{equation}\label{eq29}
\left| \lim_{n \to \infty} \frac{x_n}{r_n} - \lim_{n \to \infty} \frac{y_n}{r_n} \right|= \lim_{n \to \infty} \frac{|x_n- y_n|}{r_n},
\end{equation}
we obtain the logical equivalence
\begin{equation}\label{eq37n}
\left( \lim_{n \to \infty} \frac{x_n}{r_n}= \lim_{n \to \infty} \frac{y_n}{r_n} \right) \Leftrightarrow (\tilde{x}\thicksim \tilde{y})
\end{equation}
for all $\tilde{x}, \tilde{y} \in \tilde{E}_{0, \tilde{r}}$.
\end{proof}

\begin{corollary}\label{cor37}
Let $0 \in E \subseteq \mathbb{R}^+$ and let $\tilde{E}_{0, \tilde{r}}$ be a maximal self-stable family w.r.t. $\tilde{r}=(r_n)_{n \in \mathbb{N}}$. Then a sequence $\tilde{x}=(x_n)_{n \in \mathbb{N}} \in \tilde{E}$ belongs to $\tilde{E}_{0, \tilde{r}}$ if and only if there exists a finite limit $\lim_{n \to \infty}\frac{x_n}{r_n}$.
\end{corollary}
\begin{proof}
As was shown in the proof of Proposition \ref{p9} the statement $\tilde{x} \in \tilde{E}_{0, \tilde{r}}$ implies the existence of finite $\lim_{n \to \infty} \frac{x_n}{r_n}$ . The converse follows from \eqref{eq37n}
\end{proof}
This corollary shows, in particular, that for every normalizing sequence $\tilde{r}$ and every $E \subseteq \mathbb{R}^+$ with $0 \in E$ there is a unique pretangent space $\Omega^{E}_{0,\tilde{r}}$. The last assertion generally does not hold for arbitrary metric spaces, (see \cite{adk} for details).

We can identify the metric space $(\Omega^{E}_{0, \tilde{r}}, \rho)$ with a subspace $\overline{\Omega}^{E}_{0, \tilde{r}}$ of $\mathbb{R}^+$ by the following way. For every $t \in \mathbb{R}^+$ we set $t \in \overline{\Omega}^{E}_{0, \tilde{r}}$ if and only if there is a sequence $\tilde{x} \in \tilde{E}$ such that the equality
\begin{equation*}
t = \lim_{n \to \infty} \frac{x_n}{r_n}
\end{equation*}
holds. Let us define a mapping $L :\Omega^{E}_{0, \tilde{r}} \to \overline{\Omega}^{E}_{0, \tilde{r}}$ as
\begin{equation}\label{eq211}
L(\alpha) := \lim_{n \to \infty} \frac{x_n}{r_n}
\end{equation}
where $(x_n)_{n \to \mathbb{N}}$ is an arbitrary element of $\tilde{E}_{0, \tilde{r}}$ which belongs to $\alpha$.

\begin{proposition}\label{p10}
Let $E \subseteq \mathbb{R}$ with $0 \in E$, let $\tilde{r}=(r_n)_{n \in \mathbb{N}}$ be a normalizing sequence and let $\Omega^{E}_{0, \tilde{r}}$ be the corresponding pretangent space. Then the mapping $L :\Omega^{E}_{0, \tilde{r}} \to \overline{\Omega}^{E}_{0, \tilde{r}}$ defined by \eqref{eq211} is an isometric bijection between $\Omega^{E}_{0, \tilde{r}}$ and $\overline{\Omega}^{E}_{0, \tilde{r}}$ satisfying the equality
\begin{equation*}
L(\alpha_0)=0
\end{equation*}
where $\alpha_0$ is an element of $\Omega^{E}_{0, \tilde{r}}$ containing the constant sequence $\tilde{0}$. Moreover, if $A \subseteq \mathbb{R}^+$, $0 \in A$ and $F:\Omega^E_{0, \tilde{r}} \to A$ is an isometric bijection such that $F(\alpha_0)=0$, then $A = \overline{\Omega}^{E}_{0, \tilde{r}}$.
\end{proposition}
\begin{proof}
It follows from Corollary \ref{cor37} that $L$ is bijective. Equalities \eqref{eq27}, \eqref{eq28} and \eqref{eq211} imply that
\begin{equation*}
\rho(\alpha, \beta)= |L(\alpha)- L(\beta)|
\end{equation*}
for all $\alpha, \beta \in \Omega^{E}_{0, \tilde{r}}$. Hence $L$ is isometric. Now if $A \subseteq \mathbb{R}^+$, $0 \in A$ and $F:\Omega^{E}_{0, \tilde{r}} \to A$ is an isometric bijection such that $F(\alpha_0)=0$, then for every $x \in A$ there is a unique $\beta \in \Omega^{E}_{0,\tilde{r}}$ such that $x= F(\beta)$. Moreover we have $x= |x-0|= \rho(\beta, \alpha_0)$. Hence
\begin{equation*}
A= \{ \rho(\alpha_0, \beta): \overline{\Omega}^{E}_{0, \tilde{r}} \}.
\end{equation*}
In particular, the equality $\overline{\Omega}^{E}_{0, \tilde{r}} = \{ \rho(\alpha_0, \beta): \beta \in \overline{\Omega}^{E}_{0, \tilde{r}}  \}$ is valid. Thus $A = \overline{\Omega}^{E}_{0, \tilde{r}}$ holds.
\end{proof}

\begin{proposition}\label{prop218}
Let $0 \subseteq \mathbb{R}^+$. Then for every normalizing sequence $\tilde{r}=(r_n)_{n \in \mathbb{N}}$ the pretangent space $\Omega^{E}_{0, \tilde{r}}$ is complete.
\end{proposition}
\begin{proof}
It suffices to show that $\overline{\Omega}^{E}_{0, \tilde{r}}$ is a closed subset of $\mathbb{R}^+$ for every $\tilde{r}$. Suppose the contrary and choose a point $x \in \mathbb{R}^+$ such that
\begin{equation*}
x \in ac \overline{\Omega}^{E}_{0, \tilde{r}} ~\text{and}~ x \notin \overline{\Omega}^{E}_{0, \tilde{r}}.
\end{equation*}
Then there is a sequence $(x_m)_{m \in \mathbb{N}}$ with $x= \lim_{m \to \infty}x_m$ and $x_m \in \overline{\Omega}^{E}_{0, \tilde{r}}$ and $|x- x_{m+1}| < |x-x_m|$ for every $m \in \mathbb{N}$. For every $m \in \mathbb{N}$ we can find a sequence $(x_{n,m})_{n \in \mathbb{N}} \in \tilde{E}$ satisfying the conditions
\begin{equation*}
x_m = \lim_{n \to \infty} \frac{x_{n,m}}{r_n} ~\text{and}~\left| \frac{x_{n,m}}{r_n}- x_m \right| < |x_m-x|
\end{equation*}
for all $n \in \mathbb{N}$. Using the last inequality with $m=n$ we obtain
\begin{equation*}
\left| \frac{x_{n,n}}{r_n}- x_n \right| < |x_n-x|.
\end{equation*}
Hence
\begin{equation*}
\lim_{n \to \infty} \left| \frac{x_{n,n}}{r_n}- x_n \right| =0
\end{equation*}
that implies
\begin{equation*}
\lim_{n \to \infty} \frac{x_{n,n}}{r_n} = \lim_{n \to \infty} x_n =x.
\end{equation*}
By Corollary \ref{cor37} we have $(x_{n,n})_{n \in \mathbb{N}} \in \tilde{E}_{0, \tilde{r}}$. Thus $x \in \overline{\Omega}^{E}_{0, \tilde{r}}$, which is a contradiction.
\end{proof}

\begin{definition}\label{dkck}
Let $E$ and $T$ be subsets of $\mathbb{R}^+$. We shall write $E \preceq T$ if for every sequence $(e_n)_{n \in \mathbb{N}}$ with $\lim_{n \to \infty}e_n = 0$ and $e_n \in E \backslash \{0\}$, $(e_n)_{n \in \mathbb{N}} \in \tilde{E}$, there is a sequence $(t_n)_{n \in \mathbb{N}}$, such that
\begin{equation*}
\lim_{n \to \infty} \frac{e_n}{t_n}=1
\end{equation*}
and $t_n \in T \backslash \{0\}$ for every $n \in \mathbb{N}$.
\end{definition}

\begin{proposition}\label{p11}
Let $E$ and $T$ be subsets of $\mathbb{R}^+$, $0 \in E \cap T$ and let $\tilde{r}$ be a normalizing sequence. If $E \preceq T$ and $T \preceq E$, then the equality
\begin{equation}\label{eq212}
\overline{\Omega}^{E}_{0, \tilde{r}} = \overline{\Omega}^{T}_{0, \tilde{r}}
\end{equation}
holds.
\end{proposition}
\begin{proof}
Let $E \preceq T$ and $T \preceq E$ hold. These conditions imply
\begin{equation*}
(0 \in acT) \Leftrightarrow (0 \in acE).
\end{equation*}
If $0 \notin acT$ and $0 \notin acE$, then we evidently obtain
\begin{equation*}
\overline{\Omega}^{E}_{0, \tilde{r}}= \{0\} = \overline{\Omega}^{T}_{0, \tilde{r}}.
\end{equation*}
Now let $0 \in acT$ and $0 \in acE$. If $t \in \overline{\Omega}^{T}_{0, \tilde{r}}$ and $t \neq 0$, then by Corollary \ref{cor37} there is $(t_n)_{n \in \mathbb{N}} \in \tilde{T}$ such that
\begin{equation}\label{eq213}
t = \lim_{n \to \infty} \frac{t_n}{r_n}
\end{equation}
and $t_n \in T \backslash \{0\}$ for all $n$. From $T \preceq E$ it follows that there is $(s_n)_{n \in \mathbb{N}} \in \tilde{E}$ such that
\begin{equation}\label{eq214}
\lim_{n \to \infty} \frac{s_n}{t_n}=1.
\end{equation}
Limit relations \eqref{eq213} and \eqref{eq214} imply that
\begin{equation*}
t= \lim_{n \to \infty}\frac{s_n}{r_n}.
\end{equation*}
Hence $t \in \overline{\Omega}^{E}_{0, \tilde{r}}$. Thus we have the inclusion $\overline{\Omega}^{T}_{0, \tilde{r}} \subseteq \overline{\Omega}^{E}_{0, \tilde{r}}$. Using the statement $E \preceq T$ we obtain the inclusion $\overline{\Omega}^{E}_{0, \tilde{r}} \subseteq \overline{\Omega}^{T}_{0, \tilde{r}}$. Equality \eqref{eq212} follows.
\end{proof}
\begin{corollary}\label{c3}
Let $0 \in E$ and $E \subseteq \mathbb{R}^+$ and let $\overline{E}$ be the closure of $E$. Then the equality
\begin{equation*}
\overline{\Omega}^{E}_{0, \tilde{r}} = \overline{\Omega}^{\overline{E}}_{0, \tilde{r}}
\end{equation*}
holds for every normalizing sequence $\tilde{r}$.
\end{corollary}

\begin{remark}\label{r2}
If equality \eqref{eq212} holds for every normalizing sequence $\tilde{r}$ , then it can be proved that $E \preceq T$ and $T \preceq E$. Similar results are valid for subspaces of arbitrary metric spaces, (see \cite{dovgoshey3}).
\end{remark}
Let $(n_k)_{k \in \mathbb{N}}$ be an infinite strictly increasing sequence of natural numbers. Let $\tilde{r}'= (r_{n_k})_{k \in \mathbb{N}}$ be the corresponding subsequence of a normalizing sequence $\tilde{r}=(r_n)_{ \in \mathbb{N}}$. Define a subset $\tilde{E}_{0, \tilde{r}'}$ of a set $\tilde{E}$ by the rule
\begin{equation*}
\left( (x_n)_{n \in \mathbb{N}} \in \tilde{E}_{0, \tilde{r}'} \right) \Leftrightarrow \left( \lim_{k\to \infty} \frac{x_{n_k}}{r_{n_k}} < \infty \right).
\end{equation*}
One easily checks that $\tilde{E}_{0, \tilde{r}'}$ is maximal self-stable w.r.t. $\tilde{r}'$, i.e., for all $\tilde{x}, \tilde{y} \in \tilde{E}_{0, \tilde{r}'}$ there are finite limits
\begin{equation*}
|\tilde{x}- \tilde{y}|_{\tilde{r}'} = \lim_{k \to \infty} \frac{|x_{n_k}- y_{n_k}|}{r_{n_k}},
\end{equation*}
and if $\tilde{z} \in \tilde{E} \backslash \tilde{E}_{0, \tilde{r}'}$, then there exists $\tilde{x} \in \tilde{E}_{0, \tilde{r}'}$ such that the limit
\begin{equation*}
\lim_{k \to \infty} \frac{|x_{n_k}- z_{n_k}|}{r_{n_k}}
\end{equation*}
is infinite or does not exist. It is also clear that $\tilde{E}_{0,\tilde{r}} \subseteq \tilde{E}_{0, \tilde{r}'}$ and that $|.,.|_{r'}$ is a pseudometric on $\tilde{E}_{0, \tilde{r}}$ satisfying the equality
\begin{equation*}
|\tilde{x}- \tilde{y}|_{\tilde{r}} = |\tilde{x}- \tilde{y}|_{\tilde{r}'}
\end{equation*}
for all $\tilde{x}, \tilde{y} \in \tilde{E}_{0, \tilde{r}}$. Let $(\Omega_{0, \tilde{r}'}^{E}, \rho')$ be the metric identification of $(\tilde{E}_{0, \tilde{r}'}, |.,.|_{\tilde{r}'})$. Define the subset $\overline{\Omega}^{E}_{0, \tilde{r}'}$ of $\mathbb{R}^+$ and the mapping $L': \Omega^{E}_{0, \tilde{r}'} \to \overline{\Omega}^{E}_{0, \tilde{r}'} $ by the rules
\begin{equation*}
\left( t \in \overline{\Omega}^{E}_{0, \tilde{r}'} \right) \Leftrightarrow \left( \text{there is}~\tilde{x} \in \tilde{E}_{0, \tilde{r}'}~\text{with}~\lim_{k \to \infty}\frac{x_{n_k}}{r_{n_k}}=t \right)
\end{equation*}
and, respectively,
\begin{equation*}
\Omega_{0, \tilde{r}'}^{E} \ni \alpha \ni (x_n)_{n \in \mathbb{N}} \mapsto L'(\alpha) = \lim_{k \to \infty} \frac{x_{n_k}}{r_{n_k}} \in \overline{\Omega}_{0, \tilde{r}'}^{E}.
\end{equation*}
Then $L'$ is an isometric bijection such that $L'(\alpha_0')=0$ where $\alpha_0'$ is the point of $\Omega_{0, \tilde{r}'}^{E}$ which contains the constant sequence $\tilde{0}$. Moreover, it is easy to prove that the diagram
\begin{equation*}
\left.\begin{array}{ccccc}\tilde{E}_{0,r} & \overset{\pi}{\longrightarrow} & \Omega^{E}_{0,\tilde{r}} & \overset{L}{\longrightarrow} &  \overline{\Omega}^{E}_{0,\tilde{r}} \\~ & ~ & ~ & ~ & ~ \\ in_{\tilde{E}}  \downarrow & ~ & \downarrow em' & ~ & \downarrow in_{\mathbb{R}^+} \\~ & ~ & ~ & ~ & ~ \\ \tilde{E}_{0,r'} & \overset{\pi'}{\longrightarrow} & \Omega^{E}_{0,\tilde{r}'} & \overset{L'}{\longrightarrow} &  \overline{\Omega}^{E}_{0,\tilde{r}'}\end{array}\right.
\end{equation*}
is commutative, where $\pi$ and $\pi'$ are the natural projections

\begin{equation*}
\pi(\tilde{x})= \{ \tilde{y} \in \tilde{E}_{0, \tilde{r}}:|\tilde{x}- \tilde{y}|_{\tilde{r}}=0 \}, ~
\pi'(\tilde{x})= \{ \tilde{y} \in \tilde{E}_{0, \tilde{r}'}:|\tilde{x}- \tilde{y}|_{\tilde{r}'}=0 \};
\end{equation*}
$in_{\tilde{E}}$ and $in_{\mathbb{R}^+}$ are the injections,
\begin{equation*}
in_{\tilde{E}}(\tilde{x})= \tilde{x}~\text{and}~in_{\mathbb{R^+}}(t)=t;
\end{equation*}
and $em'$ is an isometric embedding for which the equality $em' \cdot \pi = in_{\tilde{E}} \cdot \pi'$ holds.

Now we are ready to describe the set of porosity numbers $\textbf{P}(E)$ on the language of pretangent spaces.

\begin{theorem}\label{th5}
Let $E \subseteq \mathbb{R}^+$ and let $0 \in E$. A number $p \in \mathbb{R}^+$ is a porosity number of the set $E$ at $0$ if and only if there are a normalizing sequence $\tilde{r}=(r_n)_{n \in \mathbb{N}}$ and an open interval $(a,b) \subseteq (0,1)$ with $|a-b|=p$ which satisfy the following conditions.

  ($\textit{i}$) The equality
  \begin{equation}\label{eq215}
  (a,b) \cap \overline{\Omega}_{0, \tilde{r}'}^{E} = \emptyset
  \end{equation}
  holds for every subsequence $\tilde{r}'= (r_{n_k})_{k \in \mathbb{N}}$ of $\tilde{r}$.

  ($\textit{ii}$) If $(c,d) \subseteq (0,1)$ is an open interval such that
  \begin{equation}\label{eq216}
  (c,d) \cap \overline{\Omega}_{0, \tilde{r}'}^{E} = \emptyset
  \end{equation}
  holds for every subsequence $\tilde{r}'$ of $\tilde{r}$, then $|c-d| \leq |a-b|$.

\end{theorem}
\begin{proof}
If $0 \notin acE$, then the set $\textbf{P}(E)$ of porosity numbers contains only the number $1$ and the equality $\overline{\Omega}_{0, \tilde{r}'}^{E}= \{ 0 \}$ holds for every normalizing sequence $\tilde{r}$. Hence the theorem is trivially true when $0 \notin acE$.

Let us consider the case $0 \in acE$. It follows from Corollary \ref{c3} that we can assume that $E$ is a closed set. Let $p$ be a porosity number of $E$. Then there is a sequence $\tilde{h}= (h_m)_{m \in \mathbb{N}}$ such that $\lim_{m \to \infty}h_m =0$, $h_m > 0$ for every $m \in \mathbb{N}$, and
\begin{equation}\label{eq217}
p = \lim_{m \to \infty} \frac{\lambda(E, h_m)}{h_m}
\end{equation}
where $\lambda(E, h_m)$ is the same as in Definition \ref{d1}.

For every $m \in \mathbb{N}$, let $(a_m, b_m)$ be the largest open interval in $(0, h_m)$ such that
\begin{equation}\label{eq218}
(a_m, b_m) \cap E = \emptyset.
\end{equation}
(This interval can be empty, $a_m = b_m$, if $p=0$.) Then, by definition, we have
\begin{equation}\label{eq219}
\lambda(E, h_m) = b_m- a_m
\end{equation}
and, in addition, $a_m \in E$ because $E$ is closed. Suppose also that there is a subsequence $\tilde{h}'=(h_{m_n})_{n \in \mathbb{N}}$ of $\tilde{h}$ such that
\begin{equation}\label{eq2.20*}
b_{m_n} \in E ~\text{for every}~ n \in \mathbb{N}.
\end{equation}
Passing this subsequence we can also assume the existence the finite limit
\begin{equation}\label{eq220}
b= \lim_{n \to \infty} \frac{b_{m_n}}{h_{m_n}}.
\end{equation}
For every $n \in \mathbb{N}$ denote by $r_n$ the element $h_{m_n}$ of $\tilde{h}$ and consider the pretangent space $\Omega_{0, \tilde{r}}^{E}$ w.r.t. the normalizing sequence $\tilde{r}=(r_n)_{n \in \mathbb{N}}$. It is clear that $a \in \overline{\Omega}_{0, \tilde{r}'}^{E}$ and $b \in \overline{\Omega}_{0, \tilde{r}}^{E}$ hold for every subsequence $\tilde{r}'$ of $\tilde{r}$. Equalities \eqref{eq217}, \eqref{eq219} and \eqref{eq220} give us the equality $|a-b|=p$. Let us prove \eqref{eq215} for every $\tilde{r}'$.

Suppose contrary that there is a subsequence $\tilde{r}'= (r_{n_k})_{k \in \mathbb{N}}$ of $\tilde{r}$ such that
\begin{equation*}
(a,b) \cap \overline{\Omega}_{0, \tilde{r}'}^{E} \neq \emptyset.
\end{equation*}
Let $x \in \overline{\Omega}_{0, \tilde{r}'}^{E}$ with
\begin{equation}\label{eq221}
a<x<b.
\end{equation}
By definition of $\overline{\Omega}_{0, \tilde{r}'}^{E}$ we can find $(x_n)_{n \in \mathbb{N}} \in \tilde{E}$ such that
\begin{equation*}
x = \lim_{k \to \infty}\frac{x_{n_k}}{r_{n_k}}.
\end{equation*}
Now using \eqref{eq221} we obtain
\begin{equation*}
\lim_{k \to \infty} \frac{a_{n_k}}{r_{n_k}} < \lim_{k \to \infty} \frac{x_{n_k}}{r_{n_k}} < \lim_{k \to \infty} \frac{b_{n_k}}{r_{n_k}}.
\end{equation*}
The last double inequality implies that
\begin{equation}\label{eq222}
x_{n_k} \in (a_{n_k}, b_{n_k})
\end{equation}
holds for all sufficiently large $k$. Since $x_{n_k} \in E$, statement \eqref{eq222} contradicts the condition
\begin{equation*}
(a_n, b_n) \cap E = \emptyset ~\text{for every}~n \in \mathbb{N}.
\end{equation*}
Hence \eqref{eq215} holds for every $\tilde{r}'$.

To prove ($\textit{ii}$), suppose that, on the contrary, there exists $\tilde{r}'$ such that for an interval $(c,d)$ with $0 \leq c < d \leq 1$, we have \eqref{eq216} but $|c-d| > |a-b|$. Without loss of generality we can assume that $(c,d)$ is the largest open subinterval of $(0,1)$ which satisfies condition ($\textit{ii}$). Then we have either
\begin{equation}\label{y1}
c \in \overline{\Omega}^{E}_{0, \tilde{r}'}~ \text{and}~d \in \overline{\Omega}^{E}_{0, \tilde{r}'}
\end{equation}
or
\begin{equation}\label{y2}
c \in \overline{\Omega}^{E}_{0, \tilde{r}'}~\text{and}~d=1,
\end{equation}
because, by Proposition \ref{prop218}, the set $\overline{\Omega}^{E}_{0, \tilde{r}'}$ is closed in $\mathbb{R}^+$. In what follows we assume that \eqref{y1} holds. (The case where \eqref{y2} holds can be considered similarly.)

Let $\varepsilon > 0$ be a number for which
\begin{equation}\label{eq223}
c \leq (1+ \varepsilon) c < (1- \varepsilon) d < d ~\text{and}~ |(1+ \varepsilon)c - (1-\varepsilon)d| > |a-b|=p.
\end{equation}
Since by \eqref{y1} we have $c,d \in \overline{\Omega}_{0, \tilde{r}}^{E}$, there are $(c_n)_{n \in \mathbb{N}} \in \tilde{E}$ and $(d_n)_{n \in \mathbb{N}} \in \tilde{E}$ such that
\begin{equation*}
c = \lim_{k \to \infty} \frac{c_{n_k}}{r_{n_k}}~\text{and}~ d = \lim_{k \to \infty} \frac{d_{n_k}}{r_{n_k}}.
\end{equation*}
Using \eqref{eq223} we obtain the inequality
\begin{equation}\label{eq224}
|(1+ \varepsilon)c_{n_k}- (1- \varepsilon)d_{n_k}| > \lambda(E, r_{n_k})
\end{equation}
for all sufficiently large $k$. Consequently there is a point $x_{n_k}$ such that
\begin{equation*}
x_{n_k} \in \left( (1+ \varepsilon)c_{n_k}, (1- \varepsilon)d_{n_k} \right) \cap E.
\end{equation*}
Passing to a subsequence we can assume that there exists a finite limit
\begin{equation}\label{eq225}
x= \lim_{k \to \infty} \frac{x_{n_k}}{r_{n_k}}.
\end{equation}
From \eqref{eq225} we obtain $x \in [(1+ \varepsilon)c_{n_k}, (1- \varepsilon)d_{n_k}]$. The inclusion
\begin{equation*}
\left[(1+ \varepsilon)c_{n_k}, (1- \varepsilon)d_{n_k}\right] \subseteq (c,d)
\end{equation*}
implies
\begin{equation*}
x \in (c,d).
\end{equation*}
From \eqref{eq225} it follows that $x \in \overline{\Omega}_{0, \tilde{r}'}^{E}$. Hence $x \in (c,d) \cap \overline{\Omega}_{0, \tilde{r}'}^{E}$, contrary to \eqref{eq216}.

Thus if $p$ is a porosity number and \eqref{eq2.20*} holds, then conditions ($\textit{i}$) and ($\textit{ii}$) are satisfied. If there is no strictly increasing $(m_n)_{n \in \mathbb{N}}$ so that \eqref{eq220} hold. Then $b_m = h_m$, and $h_m \notin E$ for all sufficiently large $m$. This case is more simple and can be considered similarly.

Suppose now that $\tilde{r}=(r_n)_{n \in \mathbb{N}}$ is a normalizing sequence such that conditions ($\textit{i}$) and ($\textit{ii}$) take place with an interval $(a,b) \subseteq (0,1)$. We must prove that $p= |a-b|$ is a porosity number of $E$ at $0$. Let us consider a sequence $\left( \frac{\lambda(E, r_n)}{r_n} \right)_{n \in \mathbb{N}}$. This sequence contains a convergent subsequence $\left( \frac{\lambda(E, r_{n_k})}{r_{n_k}} \right)_{k \in \mathbb{N}}$. By definition
\begin{equation*}
p^{*} := \lim_{k \to \infty}\frac{\lambda(E, r_{n_k})}{r_{n_k}}
\end{equation*}
is a porosity number of $E$. It is sufficient to prove that $p^{*}= |a-b|$. Write $\tilde{r}^{*}=(r_{n_k})_{k \in \mathbb{N}}$. It was shown in the first part of the proof that there is a subsequence $\tilde{r}^{**}$ of $\tilde{r}^{*}$ and an interval $(a^{*}, b^{*})$ such that ($\textit{i}$) and ($\textit{ii}$) hold for every subsequence $\tilde{r}'$ of $\tilde{r}^{**}$. In particular we have $p^{*} = |a^*- b^*|$. Using condition ($\textit{ii}$) we obtain the inequalities
\begin{equation*}
|a^*- b^*| \leq |a-b| ~\text{and}~ |a-b| \leq |a^*- b^*|.
\end{equation*}
Hence $p= |a-b|= |a^*- b^*| = p^{*}$. Thus $p= |a-b|$ is a porosity number of $E$ as required.
\end{proof}

\begin{corollary}\label{cor29}
Let $E$ and $T$ be subsets of $\mathbb{R}^+$. If the conditions $E \preceq T$ and $T \preceq E$ hold, then $E$ and $T$ have equal sets of porosity numbers,
\begin{equation}\label{eqkck}
\textbf{P}(E) = \textbf{P}(T).
\end{equation}
\end{corollary}

\begin{proof}
It follows directly from Theorem \ref{th5} and Proposition \ref{p11} if $0 \in E$ and $0 \in T$. Otherwise, it suffices to note that $\textbf{P}(X) = \textbf{P}(X \cup \{0\})$ for every $X \subseteq \mathbb{R}^+$.
\end{proof}

\begin{corollary}\label{cor30}
Let $E \subseteq \mathbb{R}^+$ and let $0 \in E$. The set $E$ is strongly porous at $0$ if and only if there is a normalizing sequence $\tilde{r}=(r_n)_{n \in \mathbb{N}}$ such that the equality
\begin{equation*}
(0,1) \cap \overline{\Omega}^{E}_{0, \tilde{r}'}= \emptyset
\end{equation*}
holds for every subsequence $\tilde{r}'$ of the sequence $\tilde{r}$.
\end{corollary}

\begin{proof}
Let $(a,b) \subseteq (0,1)$ and $|a-b|=1$. Then it is easy to see that $a=0$ and $b=1$. Note also that $E$ is strongly porous if and only if $1 \in \textbf{P}(E)$. Now it suffices to use Theorem \ref{th5} with $p=1$ and $(a,b)= (0,1)$.
\end{proof}

Using Theorem 4 from \cite{dak} we can also give another description of the strongly porous at $0$ subsets of $\mathbb{R}^+$.

\begin{corollary}\label{cor31}
Let $E \subseteq \mathbb{R}^+$ and let $0 \in E$. The set $E$ is strongly porous at $0$ if and only if there is a normalizing sequence $\tilde{r}$ such that $\Omega^{E}_{0, \tilde{r}'}$ is one-point for every subsequence $\tilde{r}'$ of $\tilde{r}$.
\end{corollary}

\section{The Upper Porosity At Infinity}

Let $\mu :\mathbb{N} \to \mathbb{R}^+$ be a strictly decreasing function such that $\lim_{n \to \infty} \mu (n) =0$ and let $E \subseteq \mathbb{N}$. Let us define the numbers $\overline{p_{\mu}}(E)$ and $\underline{p_{\mu}}(E)$ as
\begin{equation}\label{pm}
\overline{p}_{\mu} (E) =\limsup_{n \to \infty} \frac{\lambda_{\mu} (E,n)}{\mu (n)}
\end{equation}
and
\begin{equation}\label{eq5}
\underline{p}_{\mu}(E)=\liminf_{n \to \infty}  \frac{\lambda_{\mu}(E,n)}{\mu(n)}
\end{equation}
where
\begin{equation}\label{eq6}
\lambda_{\mu} (E,n) = \sup \{ |\mu (n^{(1)})- \mu (n^{(2)})|:n \leq n^{(1)} < n^{(2)}, ~(n^{(1)}, n^{(2)}) \cap E = \emptyset \}.
\end{equation}
We will say that $\mu$ is a scaling function and that $\overline{p}_{\mu}(E)$ and $\underline{p}_{\mu}(E)$ are the upper porosity of $E$ at infinity and, respectively, the lower porosity of $E$ at infinity.

\begin{remark}
If $E$ is an infinite subset of $\mathbb{N}$, $|E|= \infty$, then, for every $n \in \mathbb{N}$, $\lambda_{\mu}(E,n)$ is the length of the largest open subinterval of $(0, \mu(n))$ that contains no point of $\mu(E)$ and has a form $(\mu(n^{(2)}), \mu(n^{(1)}))$ with $n^{(1)} < n^{(2)}$. For the case of finite $E$ we evidently have $\lambda_{\mu}(E,n) = \mu(n)$ for all sufficiently large $n$. Consequently the equalities
\begin{equation*}
\underline{p}_{\mu}(E) = \overline{p}_{\mu}(E) = 1
\end{equation*}
hold with every scaling function $\mu$ for all $E \subseteq \mathbb{N}$ with $|E|< \infty$.
\end{remark}

For given $\mu$ we define $E \subseteq \mathbb{N}$ to be: porous at infinity if $\overline{p}_{\mu}(E) > 0$, strongly porous at infinity if $\overline{p}_{\mu}(E)=1$ and, respectively, nonporous at infinity if $\overline{p}_{\mu}(E) = 0$.

\begin{definition}\label{d3}
Let $\mu:\mathbb{N} \to \mathbb{R}^+$ be a scaling function and let $E \subseteq \mathbb{N}$. A number $p$ is a porosity number of $E$ at infinity if there is a strictly increasing sequence $(n_k)_{k \in \mathbb{N}}$ such that
\begin{equation}\label{eq7}
p= \lim_{k \to \infty} \frac{\lambda_{\mu}(E, n_k)}{\mu(n_k)}
\end{equation}
where $\lambda_{\mu}(E, n_k)$ is defined by \eqref{eq6}.
\end{definition}

For $E \subseteq \mathbb{N}$ and a scaling function $\mu$ we shall write $\textbf{P}_{\mu}(E)$ for the set of all porosity numbers of $E$ at infinity. It is clear that $\textbf{P}_{\mu}(E)$ is a closed subset of $[0,1]$ and
\begin{equation*}
\overline{p}_{\mu}(E) = \max_{p \in \textbf{P}_{\mu}(E)}p~~~\text{and}~~~\underline{p}_{\mu}(E) = \min_{p \in \textbf{P}_{\mu}(E)}p.
\end{equation*}

The following lemma often allows us to compute the upper porosity at infinity for subsets of $\mathbb{N}$ which are defined by recurrence relations.

\begin{lemma}\label{l1}
Let
\begin{equation*}
E = \{ n_1, n_2, ..., n_k, n_{k+1},... \} \subseteq \mathbb{N}
\end{equation*}
where $n_k < n_{k+1}$ for every $k \in \mathbb{N}$. Then for each scaling function $\mu$ the upper porosity $\overline{p}_{\mu}(E)$ satisfies the equality
\begin{equation}\label{eq17}
\overline{p}_{\mu}(E) = 1- \liminf_{k \to \infty} \frac{\mu(n_k)}{\mu(n_{k-1})}.
\end{equation}
\end{lemma}

\begin{proof}
Let $\mu :\mathbb{N} \to \mathbb{R}^+$ be a scaling function. The right-hand side of \eqref{eq17} is less than or equal to $\overline{p}_{\mu}(E)$. Indeed,
\begin{eqnarray*}
1- \liminf_{k \to \infty} \frac{\mu(n_k)}{\mu(n_{k-1})} &=& \limsup_{k \to \infty} \frac{\mu(n_{k-1})- \mu(n_k)}{\mu(n_{k-1})} \\
& \leq & \limsup_{k \to \infty} \frac{\lambda_{\mu}(E, n_k)}{\mu(n_k)} \leq \limsup_{n \to \infty} \frac{\lambda_{\mu}(E, n)}{\mu(n)} = \overline{p}_{\mu}(E).
\end{eqnarray*}
Hence to prove \eqref{eq17} it suffices to show that
\begin{equation}\label{eq18}
\overline{p}_{\mu}(E) \leq 1- \liminf_{k \to \infty} \frac{\mu(n_k)}{\mu(n_{k-1})}.
\end{equation}
For every $n \in \mathbb{N}$ let $n^{(1)}$ and $n^{(2)}$ be positive integer numbers such that $n \leq n^{(1)} < n^{(2)}$ and
\begin{equation*}
\lambda_{\mu}(E, n)= \mu(n^{(1)})- \mu(n^{(2)}).
\end{equation*}
Since $\mu$ is decreasing, the inequality
\begin{equation*}
\frac{\lambda_{\mu}(E,n)}{\mu(n)} \leq \frac{\lambda_{\mu}(E,n)}{\mu(n^{(1)})}
\end{equation*}
holds. Moreover it is easy to see that
\begin{equation*}
\lambda_{\mu}(E,n) = \lambda_{\mu}(E,n^{(1)}).
\end{equation*}
Consequently we have
\begin{equation}\label{eq19}
\frac{\lambda_{\mu}(E,n)}{\mu(n)} \leq \frac{\lambda_{\mu}(E,n^{(1)})}{\mu(n^{(1)})} = 1- \frac{\mu(n^{(2)})}{\mu(n^{(1)})}
\end{equation}
for every $n \in \mathbb{N}$. Since for every $n \in \mathbb{N}$ there is $k \in \mathbb{N}$ such $n^{(2)}= n_{k+1}$ and $n^{(1)}=n_k$, where $n_{k+1}, n_k \in E$, it follows from \eqref{eq19} that
\begin{equation*}
\limsup_{n \to \infty} \frac{\lambda_{\mu}(E,n)}{\mu(n)} \leq \limsup_{k \to \infty} \left( 1- \frac{\mu(n_{k+1})}{\mu(n_k)} \right),
\end{equation*}
which implies \eqref{eq18}.
\end{proof}

The following theorem can be used as a base for translation of some results related to the classical right upper porosity at $0$ into the language of the upper porosity at infinity.

\begin{theorem}\label{th2}
Let $E \subseteq \mathbb{N}$. Then the equality
\begin{equation}\label{eq23}
\overline{p}(\mu(E)) = \overline{p}_{\mu}(E)
\end{equation}
holds for every scaling function $\mu :\mathbb{N} \to \mathbb{R}^+$.
\end{theorem}
\begin{proof}
Equality \eqref{eq23} is trivial if $|E|< \infty$. Suppose that $E$ is infinite,
\begin{equation*}
 E = \{ n_1,..., n_k, n_{k+1},... \}
\end{equation*}
where $n_k < n_{k+1}$ for every $k \in \mathbb{N}$. Note that the inequality $\overline{p}(\mu(E)) \geq \overline{p}_{\mu}(E)$ follows immediately from the definitions. In order to prove the converse inequality
\begin{equation*}\label{eq24}
\overline{p}(\mu(E)) \leq \overline{p}_{\mu}(E),
\end{equation*}
it is sufficient to show that for every $h \in (0, \mu(n_1))$ there is $n_k = n_k(h) \in E$ satisfying the inequality
\begin{equation}\label{eq25}
\frac{\lambda(\mu(E), h)}{h} \leq \frac{\mu(n_{k-1})- \mu(n_k)}{\mu(n_{k-1})}.
\end{equation}
Indeed, using Lemma \ref{l1} and \eqref{eq25} we obtain
\begin{eqnarray*}
\overline{p}(\mu(E)) &=& \limsup_{h \to \infty} \frac{\lambda(\mu(E), h)}{h} \\
& \leq & \limsup_{k \to \infty} \frac{\mu(n_{k-1})- \mu(n_k)}{\mu(n_{k-1})}= \overline{p}_{\mu}(E).
\end{eqnarray*}
We now turn to the proof of inequality \eqref{eq25}. This is trivial if $\lambda(\mu(E),h)=0$. Hence we may suppose that $\lambda(\mu(E),h)>0$. Let $h \in (0, \mu(n_1))$ and let $x,y$ be positive numbers such that $0 < x < y \leq h$ and
\begin{equation}\label{eq26}
y-x = \lambda (\mu(E), h)
\end{equation}
and
\begin{equation*}
(x,y) \cap \mu(E) = \emptyset.
\end{equation*}
Let us define $k = k(h) \in \mathbb{N}$ by the rule
\begin{equation*}
k = \min \{ j \in \mathbb{N} :\mu(n_j) \leq x, n_j \in E \}.
\end{equation*}
It follows directly from the definition of $k=k(h)$ that
\begin{equation*}
\mu(n_k) \leq x < \mu(n_{k-1}).
\end{equation*}
If $\mu(n_k) < x$, then we have
\begin{equation*}
(\mu(n_k), y) = (\mu(n_k), x] \cup (x,y) \subseteq (\mu(n_k), \mu(n_{k-1})) \cup (x,y).
\end{equation*}
 Consequently
\begin{equation*}
(\mu(n_k), y) \cap \mu(E) = \left( (\mu(n_k), \mu(n_{k-1})) \cap \mu(E) \right) \cup \left( (x,y) \cap \mu(E) \right)  = \emptyset
\end{equation*}
and
\begin{equation*}
|y- \mu(n_k)| = |x- \mu(n_k)| + |x-y|
\end{equation*}
hold. The last equality and \eqref{eq26} imply
\begin{equation*}
|y- \mu(n_k)| > \lambda(\mu(E), h),
\end{equation*}
contrary to the definition of $\lambda(\mu(E), h)$. Hence
\begin{equation*}
x= \mu(n_k)
\end{equation*}
 holds. Now the equalities
\begin{equation*}
(\mu(n_k), \mu(n_{k-1})) \cap \mu(E) = \emptyset
\end{equation*}
and
\begin{equation*}
(x,y) \cap \mu(E)=(\mu(n_k), y) \cap \mu(E) = \emptyset
\end{equation*}
imply that $y \leq \mu(n_{k-1})$. If $\mu(n_{k-1}) \leq h$, then reasoning as in the proof of the equality $x = \mu(n_k)$ we can show that
\begin{equation*}
\mu(n_{k-1})=y.
\end{equation*}
Thus in this case we have $\lambda(\mu(E), h)= \mu(n_{k-1})- \mu(n_k)$ and $\mu(n_{k-1}) \leq h$, that yields \eqref{eq25}.

In the case of $h < \mu(n_{k-1})$ we can simply obtain the equality $y=h$. Indeed if $y < h$, then
\begin{equation*}
|x-y| = |\mu(n_k)- y| < |\mu(n_k)- \mu(n_{k-1})|.
\end{equation*}
It contradicts the equality $\lambda(\mu(E),h)= |x-y|$. Consequently in the case of $h < \mu(n_{k-1})$ we have
\begin{equation*}
\frac{\lambda(\mu(E), h)}{h}= \frac{y-x}{h} = \frac{h - \mu(n_k)}{h}.
\end{equation*}
Since the function
\begin{equation*}
f(t)= \frac{t- \mu(n_k)}{t}
\end{equation*}
is increasing on $(\mu(n_k), \infty)$, the inequality $h < \mu(n_{k-1})$ implies
\begin{equation*}
\frac{h- \mu(n_k)}{h} \leq \frac{\mu(n_{k-1})- \mu(n_k)}{\mu(n_{k-1})},
\end{equation*}
that is equivalent to \eqref{eq25}.
\end{proof}
The next result describes some sufficient and necessary conditions under which $\mathbb{N}$ is nonporous or porous or strongly porous at infinity.

\begin{proposition}\label{th3}
Let $E$ be an infinite subset of $\mathbb{N}$,
\begin{equation*}
E= \{ n_1,..., n_k, n_{k+1},... \}
\end{equation*}
where $n_k < n_{k+1}$ for every $k \in \mathbb{N}$ and let $\mu:\mathbb{N} \to \mathbb{R}^+$ be a scaling function. Then the following statements hold.

(\textit{i}) The set $E$ is nonporous at infinity w.r.t. $\mu$ if and only if
\begin{equation}\label{e1}
\lim_{k \to \infty} \frac{\mu(n_{k+1})}{\mu(n_k)} = 1.
\end{equation}

(\textit{ii}) The set $E$ is porous at infinity w.r.t. $\mu$ if and only if
\begin{equation*}
\limsup_{k \to \infty} \frac{\mu(n_{k+1})}{\mu(n_k)} < 1.
\end{equation*}

(\textit{iii}) The set $E$ is strongly porous at infinity w.r.t. $\mu$ if and only if
\begin{equation*}
\liminf_{k \to \infty} \frac{\mu(n_{k+1})}{\mu(n_k)} =0.
\end{equation*}
\end{proposition}

\begin{proof}
Let prove (\textit{i}). Suppose that \eqref{e1} holds. Then, by Lemma \ref{l1}, we obtain
\begin{equation*}
\overline{p}_{\mu}(E) = 1- \liminf_{k \to \infty} \frac{\mu(n_{k+1})}{\mu(n_k)}= 1- \lim_{k \to \infty}\frac{\mu(n_{k+1})}{\mu(n_k)}=0.
\end{equation*}
Thus $E$ is nonporous at infinity w.r.t. $\mu$. Conversely, assume $\overline{p}_{\mu}(E)=0$. Then using Lemma \ref{l1} again we find that
\begin{equation*}
1= \liminf_{k \to \infty} \frac{\mu(n_{k+1})}{\mu(n_k)}.
\end{equation*}
Since $\mu$ is decreasing the inequality
\begin{equation*}
\limsup_{k \to \infty} \frac{\mu(n_{k+1})}{\mu(n_k)} \leq 1
\end{equation*}
holds. Thus
\begin{equation*}
1= \liminf_{k \to \infty} \frac{\mu(n_{k+1})}{\mu(n_k)} \leq \limsup_{k \to \infty } \frac{\mu(n_{k+1})}{\mu(n_k)} \leq 1,
\end{equation*}
that implies \eqref{e1}. Statement (\textit{i}) follows.

Statement (\textit{ii}) and (\textit{iii}) can be proved similarly.
\end{proof}

Using the pretangent spaces we can give a simple geometric characterization of subsets of $\mathbb{N}$ which are nonporous at infinity.

\begin{theorem}\label{teo211}
Let $E$ be a subset of $\mathbb{N}$, let $\mu:\mathbb{N} \to \mathbb{R}^+$ be a scaling function and let $E_{\mu}=\mu(E) \cup \{0\}$. Then the following statements are equivalent.

(\textit{i}) The set $E$ is nonporous at infinity w.r.t. $\mu$,
\begin{equation}\label{e2}
\overline{p}_{\mu}(E)=0.
\end{equation}

(\textit{ii}) The equality
\begin{equation}\label{e3}
\overline{\Omega}^{E_{\mu}}_{0, \tilde{r}}= \mathbb{R}^+
\end{equation}
holds for every normalizing sequence $\tilde{r}$.

(\textit{iii}) For every normalizing sequence $\tilde{r}$ there is a subsequence $\tilde{r}'$ such that the pretangent space $\overline{\Omega}^{E_{\mu}}_{0, \tilde{r}'}$ includes a dense subset of $(0,1)$.
\end{theorem}

\begin{proof}
Let $E$ be nonporous at infinity. Then $E$ is infinite
\begin{equation*}
E= \{ n_1,..., n_k, n_{k+1},... \}
\end{equation*}
where $n_k < n_{k+1}$ for every $k \in \mathbb{N}$. Proposition \ref{th3} implies that
\begin{equation}\label{e4}
\lim_{k \to \infty} \frac{\mu(n_{k+1})}{\mu(n_k)}=1.
\end{equation}
Let $\tilde{h}=\{ h_m \}_{m \in \mathbb{N}}$ be an arbitrary sequence of positive real numbers such that $\lim_{m \to \infty} h_m=0$. For every $m \in \mathbb{N}$ define the number $k(m)$ as
\begin{equation*}
k(m)= \min\{ k \in \mathbb{N}:\mu(n_k) \leq h_m \}.
\end{equation*}
Then the double inequality
\begin{equation}\label{e5}
\mu(n_{k(m)}) \leq h_m < \mu(n_{k(m)-1})
\end{equation}
holds for all sufficiently large $m$. It follows from \eqref{e4} and \eqref{e5} that
\begin{eqnarray*}
1 &\leq& \liminf_{m \to \infty} \frac{h_m}{\mu(n_{k(m)})} \leq \limsup_{m \to \infty} \frac{h_m}{\mu(n_{k(m)})} \\
&\leq& \limsup_{k \to \infty} \frac{\mu(n_{k(m)-1})}{\mu(n_{k(m)})}=1.
\end{eqnarray*}
Hence $\lim_{m \to \infty} \frac{h_m}{\mu(n_{k(m)})}=1$ holds. Thus we have
\begin{equation}\label{e6}
\mathbb{R}^+ \preceq \mu(E) \preceq E_{\mu}.
\end{equation}
Let $\tilde{r}=(r_n)_{n \in \mathbb{N}}$ be a normalizing sequence. By Proposition \ref{p11} the equality \eqref{e6} implies that
\begin{equation*}
\overline{\Omega}^{E_{\mu}}_{0, \tilde{r}} = \overline{\Omega}^{\mathbb{R}^+}_{0, \tilde{r}}.
\end{equation*}
Hence equality \eqref{e3} holds if and only if
\begin{equation}\label{e7}
\overline{\Omega}^{\mathbb{R}^+}_{0, \tilde{r}} = \mathbb{R}^+.
\end{equation}
To prove the last equality note that $0 \in \overline{\Omega}^{\mathbb{R}^+}_{0, \tilde{r}}$. If $s \in (0, \infty)$ and $\tilde{x}:=(s r_n)_{n \in \mathbb{N}}$, then we obviously have
\begin{equation*}
\lim_{n \to \infty} \frac{s r_n}{r_n}= s.
\end{equation*}
Hence by Corollary \ref{cor37} we obtain
$\tilde{x} \in \tilde{\mathbb{R}}^+_{0, \tilde{r}}$ where $\tilde{\mathbb{R}}^+_{0, \tilde{r}}$ is a maximal self-stable family corresponding to $\Omega^{\mathbb{R}^+}_{0, \tilde{r}}$. By Proposition \ref{p10} the statement $s \in \overline{\Omega}^{\mathbb{R}^+}_{0, \tilde{r}}$ is fulfiled. Consequently \eqref{e7} holds. The implication (\textit{i}) $\Rightarrow$ (\textit{ii}) follows. The implication (\textit{ii}) $\Rightarrow$ (\textit{iii}) is trivial. Now let (\textit{iii}) hold. Using Theorem \ref{th5} we obtain that
\begin{equation}\label{e8}
\overline{p}(E_{\mu}) =0.
\end{equation}
Since $\overline{p}(E_{\mu}) = \overline{p}(\mu(E))$, equality \eqref{e8} implies
\begin{equation*}
\overline{p}(\mu(E)) =0.
\end{equation*}
By Theorem \ref{th2} we have $\overline{p}_{\mu}(E)= \overline{p}(\mu(E))$.
Consequently \eqref{e3} holds. The implication (\textit{iii}) $\Rightarrow$ (\textit{i}) is also proved.
\end{proof}

\begin{corollary}\label{cr319}
Let $E \subseteq \mathbb{N}$, let $\mu:\mathbb{N} \to \mathbb{R}^+$ be a scaling function and let $E_{\mu}:=\mu(E) \cup \{0\}$. Then the following statements are equivalent

(\textit{i}) The set $E$ is porous at infinity w.r.t. $\mu$.

(\textit{ii}) There is a normalizing sequence $\tilde{r}$ and an interval $(a,b) \subseteq (0,1)$ with $|a-b| > 0$ such that the equality
\begin{equation*}
\overline{\Omega}^{E_{\mu}}_{0, \tilde{r}'} \cap (a,b)= \emptyset
\end{equation*}
holds for every $\tilde{r}'$.

(\textit{iii}) There is a normalizing sequence $\tilde{r}$ such that
\begin{equation*}
\mathbb{R}^+ \backslash \overline{\Omega}^{E_{\mu}}_{0, \tilde{r}'} \neq \emptyset.
\end{equation*}
\end{corollary}

\begin{proof}
It follows from Theorem \ref{th2} that the set $E$ is porous at infinity if and only if $\mu(E)$ is porous at $0$ . Note also that $\mu(E)$ is porous at $0$ if and only if $E_{\mu}$ is porous at $0$. Now using Theorem \ref{teo211} and Theorem \ref{th5} we obtain that (\textit{i}) $\Leftrightarrow$ (\textit{ii}) and (\textit{i}) $\Leftrightarrow$ (\textit{iii}).
\end{proof}

\section{The Set Of Porosity Numbers At Infinity And Relativization Of Pretangent Spaces }

To describe the set of porosity numbers at infinity, we will use a slightly modified version of concept of pretangent spaces.

Let $E$ and $M$ be subsets of $\mathbb{R}^+$ such that $0 \in E$ and $0 \in acM$. If $\tilde{r}=(r_n)_{n \in \mathbb{N}}$ is a normalizing sequence, then we write $\tilde{r} \subseteq M$ if $r_n \in M$ for every $n \in \mathbb{N}$.

It is clear that $\tilde{r} \subseteq (0, \infty)$ holds for every normalizing sequence $\tilde{r}$. Note also that $\tilde{r}' \subseteq M$ if $\tilde{r} \subseteq M$ and $\tilde{r}'$ is a subsequence of $\tilde{r}$.

\begin{theorem}\label{the1}
Let $E \subseteq \mathbb{N}$, let $\mu:\mathbb{N} \to \mathbb{R}^+$ be a scaling function and let
\begin{equation*}
E_{\mu}:= \mu(E) \cup \{ 0 \}.
\end{equation*}
A number $p$ is a porosity number of $E$ at infinity, $p \in \textbf{P}_{\mu}(E)$, if and only if there is a normalizing sequence $\tilde{t} \subseteq \mu(\mathbb{N})$ and an open interval $(a,b) \subseteq (0,1)$ such that $|a-b|=p$ and the following conditions hold.

(\textit{i}) The equality $(a,b) \cap \overline{\Omega}^{E_{\mu}}_{0, \tilde{t}'}= \emptyset$ holds for every subsequence $\tilde{t}'$ of $\tilde{t}$.

(\textit{ii}) If $(c,d) \subseteq (0,1)$ is an open interval such that $(c,d) \cap \overline{\Omega}^{E_{\mu}}_{0, \tilde{t}'}= \emptyset$ holds for every subsequence $\tilde{t}'$ of $\tilde{t}$, then $|c-d| \leq |a-b|$.
\end{theorem}

The proof of this theorem is completely similar to the proof of Theorem \ref{th5}, so that we omit it here.

In the following lemma we understand the symbol $\preceq$ in the sense of Definition \ref{dkck}.
\begin{lemma}\label{lem42}
Let $0 \in A \subseteq \mathbb{R}^+$ and let $M$ and $K$ be subsets of $\mathbb{R}^+$ which satisfy
\begin{equation*}
0 \in (acM) \cap (acK) ~\text{and}~M \preceq K.
\end{equation*}
Then for every normalizing sequence $\tilde{r}= (r_n)_{n \in \mathbb{N}} \subseteq M$ there is a normalizing sequence $\tilde{t}=(t_n)_{n \in \mathbb{N}} \subseteq K$ such that, for every strictly increasing sequence $(n_k)_{k \in \mathbb{N}}$ of natural numbers, the equality
\begin{equation}\label{e34}
\Omega^{A}_{0, \tilde{r}'} = \Omega^{A}_{0, \tilde{t}'}
\end{equation}
holds for $\tilde{r}'=(r_{n_k})_{k \in \mathbb{N}}$ and $\tilde{t}'=(t_{n_k})_{k \in \mathbb{N}}$.
\end{lemma}

\begin{proof}
Let $\tilde{r}=(r_n)_{n \in \mathbb{N}} \subseteq M$. By the definition of $\preceq$ we can find $\tilde{t}=(t_n)_{n \in \mathbb{N}} \subseteq K$ such that $\lim_{n \to \infty}\frac{r_n}{t_n}=1$. This equality gives also the limit relation $\lim_{k \to \infty} \frac{r_{n_k}}{t_{n_k}}=1$. for every strictly increasing sequence $(n_k)_{k \in \mathbb{N}}$ of natural numbers.

Consequently for every $\tilde{x}=(x_n) \in \tilde{A}$ we have
\begin{equation*}
\left(~ \text{there is} ~ \lim_{k \to \infty} \frac{x_{n_k}}{r_{n_k}} < \infty \right) \Leftrightarrow \left(~ \text{there is} ~ \lim_{k \to \infty} \frac{x_{n_k}}{t_{n_k}} < \infty \right).
\end{equation*}
Using Corollary \ref{cor37} we obtain the equality $\tilde{A}_{0, \tilde{r}'} = \tilde{A}_{0, \tilde{r}'}$ for the corresponding maximal self-stable families. Equality \eqref{e34} follows.
\end{proof}

Recall that the symbol $\textbf{P}(X)$ denotes the set of all porosity numbers at $0$ of a set $X \subseteq \mathbb{R}^+$.

The following proposition describes a sufficient condition under which $\textbf{P}(\mu(E)) = \textbf{P}_{\mu}(E)$ holds for every $E \subseteq \mathbb{N}$.

\begin{proposition}\label{p5}
Let a scaling function $\mu$ satisfy the limit relation
\begin{equation}\label{eq42}
\lim_{n \to \infty} \frac{\mu(n+1)}{\mu(n)}=1.
\end{equation}
Then for every $E \subseteq \mathbb{N}$ the equality
\begin{equation}\label{iii}
\textbf{P}(\mu(E)) = \textbf{P}_{\mu}(E)
\end{equation}
holds.
\end{proposition}

\begin{proof}
As in the proof of Theorem \ref{teo211} we obtain $(0, \infty) \preceq \mu(\mathbb{N})$ if \eqref{eq42} holds. Let $E \subseteq \mathbb{N}$. The inclusion
\begin{equation*}
\textbf{P}_{\mu}(E) \subseteq \textbf{P}_(\mu(E))
\end{equation*}
is immediate. Let $p$ be a porosity number of $\mu(E)$ at $0$. By Theorem \ref{th5} there is a normalizing sequence $\tilde{r} \subseteq (0, \infty)$ an open interval $(a,b) \subseteq (0,1)$ such that $|a-b|=p$ and statements (\textit{i}) and (\textit{ii}) of this theorem hold. Using Lemma \ref{lem42} with $A= \mu(E) \cup \{0\}$ we obtain that there is a normalizing sequence $\tilde{t} \subseteq \mu(\mathbb{N})$ such that $\tilde{t}$ and $(a,b)$ satisfy statements (\textit{i}) and (\textit{ii}) of Theorem \ref{the1}. Hence, by Theorem \ref{the1}, $p$ is a porosity number at infinity of $E$ w.r.t. $\mu$. Thus $\textbf{P}(\mu(E)) \subseteq \textbf{P}_{\mu}(E)$. Equality \eqref{iii} follows.
\end{proof}

The next theorem gives us sufficient conditions under which the sets of porosity numbers for every $E \subseteq \mathbb{N}$ coincede for two scaling functions.

\begin{theorem}\label{th1}
Let $\mu_1$ and $\mu_2$ be scaling functions. If the limit relation
\begin{equation}\label{eq10}
\lim_{n,m \to \infty} \frac{\mu_1(n)\mu_2(m)}{\mu_2(n)\mu_1(m)}=1
\end{equation}
holds, then we have the equality
\begin{equation}\label{eq11}
\textbf{P}_{\mu_1}(E) = \textbf{P}_{\mu_2}(E)
\end{equation}
for every $E \subseteq \mathbb{N}$.
\end{theorem}

\begin{proof}
If $|E|< \infty$, then \eqref{pm} and \eqref{eq5} give us the equalities
\begin{equation*}
\overline{p}_{\mu_1}(E) = \underline{p}_{\mu_1}(E) = \overline{p}_{\mu_2}(E) = \underline{p}_{\mu_2}(E)=1.
\end{equation*}
Consequently $\textbf{P}_{\mu_1}(E) = \textbf{P}_{\mu_2}(E)= \{1\}$, that implies \eqref{eq11}. Thus, without loss of generality, we assume $|E|= \infty$.

Suppose \eqref{eq10} holds. Let $p \in \textbf{P}_{\mu_1}(E)$. By Definition \ref{d3} there is a strictly increasing sequence $(n_k)_{k \in \mathbb{N}}$ such that
\begin{equation}\label{eq12}
p = \lim_{k \to \infty} \frac{\lambda_{\mu_1}(E, n_k)}{\mu_1(n_k)}.
\end{equation}
Since $E$ is infinite, for every $k \in \mathbb{N}$ there are $n^{(1)}_k, n^{(2)}_k \in \mathbb{N}$ such that $n_k \leq n^{(1)}_k < n^{(2)}_k $ and
\begin{equation}\label{eq13}
\lambda_{\mu_1}(E, n_k) = \mu_1(n^{(1)}_k)- \mu_1(n^{(2)}_k).
\end{equation}
Let $\varepsilon >0$. Equality \eqref{eq10} implies
\begin{equation}\label{eq14}
(1- \varepsilon) \frac{\mu_2(n)}{\mu_2(m)} \leq \frac{\mu_1(n)}{\mu_1(m)} \leq (1+ \varepsilon) \frac{\mu_2(n)}{\mu_2(m)}
\end{equation}
if $m$ and $n$ are sufficiently large.

Using \eqref{eq13} and \eqref{eq14} we obtain
\begin{eqnarray}\label{eq15}
\frac{\lambda_{\mu_1}(E, n_k)}{\mu_1(n_k)} &=& \frac{\mu_1(n^{(1)}_k)}{\mu_1(n_k)}- \frac{\mu_1(n^{(2)}_k)}{\mu_1(n_k)} \leq  (1+ \varepsilon) \frac{\mu_2(n^{(1)}_k)}{\mu_{(2)}(n_k)}- (1- \varepsilon) \frac{\mu_2(n^2_k)}{\mu_2(n_k)} \nonumber \\
& \leq & \frac{\mu_2(n^{(1)}_k)- \mu_2(n^{(2)}_k)}{\mu_2(n_k)} + 2 \varepsilon  \leq  \frac{\lambda_{\mu_2}(E, n_k)}{\mu_2(n_k)} + 2 \varepsilon .
\end{eqnarray}
It is clear that $\frac{\lambda_{\mu_2}(E, n_k)}{\mu_2(n_k)} \leq 1$ for every $k$. Hence $\left( \frac{\lambda_{\mu_2}(E, n_k)}{\mu_2(n_k)} \right)_{k \in \mathbb{N}}$ contains a convergent subsequence $\left( \frac{\lambda_{\mu_2}(E, n_{k'})}{\mu_2(n_{k'})} \right)_{k' \in \mathbb{N}}$. Write

\begin{equation}\label{eq16}
p' = \lim_{k' \to \infty} \frac{\lambda_{\mu_2}(E, n_{k'})}{\mu_2(n_{k'})}.
\end{equation}
It follows from \eqref{eq12}, \eqref{eq15} and \eqref{eq16} that $p \leq p' + 2 \varepsilon$. Letting $\varepsilon \to 0 $ we obtain the inequality $p \leq p'$. Let us prove the converse inequality. Using \eqref{eq16} instead of \eqref{eq12} and $(n_{k'})_{k' \in \mathbb{N}}$ instead of $(n_k)_{k \in \mathbb{N}}$ and repeating the above arguments we can find a subsequence $(n_{k''})_{k'' \in \mathbb{N}}$ of
$(n_{k'})_{k' \in \mathbb{N}}$ such that $\left( \frac{\lambda_{\mu_1}(E, n_{k''})}{\mu_1(n_{k''})} \right)_{k'' \in \mathbb{N}}$ is a convergent subsequence of $\left( \frac{\lambda_{\mu_1}(E, n_{k'})}{\mu_1(n_{k'})} \right)_{k' \in \mathbb{N}}$ and
\begin{equation*}
p' \leq \lim_{k'' \to \infty} \frac{\lambda_{\mu_1}(E, n_{k''})}{\mu_1(n_{k''})}.
\end{equation*}
Since $\left( \frac{\lambda_{\mu_1}(E, n_{k''})}{\mu_1(n_{k''})} \right)_{k'' \in \mathbb{N}}$ is also a subsequence of the sequence $\left( \frac{\lambda_{\mu_1}(E, n_k)}{\mu_1(n_k)} \right)_{k \in \mathbb{N}}$, we have the equality
\begin{equation*}
p = \lim_{k'' \to \infty } \frac{\lambda_{\mu_1}(E, n_{k''})}{\mu(n_{k''}}.
\end{equation*}
Consequently $p' \leq p$ holds, that, together with $p \leq p'$, implies the equality $p= p'$. It is clear that $p' \in \textbf{P}_{\mu_2}(E)$. Since $p$ is an arbitrary element of $\textbf{P}_{\mu_1}(E)$, we obtain $\textbf{P}_{\mu_1}(E) \subseteq \textbf{P}_{\mu_2}(E)$. Similar arguments show that $\textbf{P}_{\mu_2}(E) \subseteq \textbf{P}_{\mu_1}(E)$. Equality \eqref{eq11} follows.
\end{proof}

\begin{corollary}
Let $c >0$ and let $\mu_1:\mathbb{N} \to \mathbb{R}^+$ and $\mu_2:\mathbb{N} \to \mathbb{R}^+$ be scaling functions satisfying the equality $\lim_{n \to \infty} \frac{\mu_1(n)}{\mu_2(n)}=c$. Then the equality
\begin{equation*}
\textbf{P}_{\mu_1}(E) = \textbf{P}_{\mu_2}(E)
\end{equation*}
holds for every $E \subseteq \mathbb{N}$.
\end{corollary}

Using Lemma \ref{l1} we can prove a variant of a "week converse" to Theorem \ref{th1}.

\begin{proposition}\label{p1}
Let $\alpha \in [0,1) \cup (1, \infty]$ and let $\mu_1$ and $\mu_2$ be scaling functions. Suppose $(n_k)_{k \in \mathbb{N}}$ is a strictly increasing sequence of natural numbers such that
\begin{equation}\label{eq20}
\lim_{k \to \infty} \frac{\mu_1(n_{k+1}) \mu_2(n_k)}{\mu_1(n_k) \mu_2(n_{k+1})} = \alpha .
\end{equation}
Then for the set
$E=\{ n_k : k \in \mathbb{N} \}$ we have
\begin{equation}\label{eq21}
\overline{p}_{\mu_1}(E) \neq \overline{p}_{\mu_2}(E)
\end{equation}
or
\begin{equation*}
\overline{p}_{\mu_2}(E)=1.
\end{equation*}
\end{proposition}
\begin{proof}
Let us consider first the case $\alpha \in (1, \infty]$. Write,
\begin{equation*}
p_i^* = \liminf_{k \to \infty} \frac{\mu_i(n_{k+1})}{\mu_i(n_k)} ~\text{for}~i=1,2.
\end{equation*}
For $i=1,2$ by Lemma \ref{l1} the equality $\overline{p}_{\mu_1}(E) = \overline{p}_{\mu_2}(E)$ holds if and only if $p_1^* = p_2^*$ and, moreover,
\begin{equation*}
(\overline{p}_{\mu_2}(E)=1) \Leftrightarrow (p_2^* =0).
\end{equation*}
Suppose that $p_2^* \neq 0$. To prove \eqref{eq21} it is sufficient to show that
\begin{equation}\label{eq22}
p_1^* > p_2^*.
\end{equation}
Let $(n_{k(j)})_{j \in \mathbb{N}}$ be a subsequence of $(n_k)_{k \in \mathbb{N}}$ such that
\begin{equation*}
p_1^* = \lim_{j \to \infty} \frac{\mu_1(n_{k(j)+1})}{\mu_1(n_{k(j)})}.
\end{equation*}
This equality and \eqref{eq20} imply
\begin{equation*}
\alpha = \lim_{j\to \infty} \frac{p_1^*}{\frac{\mu_2(n_{k(j)+1})}{\mu_2(n_{k(j)})}}.
\end{equation*}
Now using the conditions $p_2^* \neq 0$ and
\begin{equation*}
p_2^* \leq \lim_{j \to \infty} \frac{\mu_2(n_{k(j)+1})}{\mu_2(n_{k(j)})}
\end{equation*}
we obtain the inequality $\alpha \leq \frac{p_1^*}{p_2^*}$. The inequality $\alpha \leq \frac{p_1^*}{p_2^*}$ and $p_2^* \neq 0$ imply \eqref{eq22}, because $\alpha \in (1, \infty]$. The case $\alpha \in [0,1)$ can be considered similarly.
\end{proof}
In the following corollary the set $E$ is the same as in Proposition \ref{p1}.
\begin{corollary}
Let $\mu_1$ and $\mu_2$ be scaling functions. If limit relation \eqref{eq20} holds with $\alpha \neq 1$ and $\overline{p}_{\mu_1}(E) = \overline{p}_{\mu_2}(E)$, then the set $E$ is strongly porous w.r.t. the both scaling functions $\mu_1$ and $\mu_2$.
\end{corollary}

Let $E$ be a subset of $\mathbb{R}^+$ and let $\mu:\mathbb{N} \to \mathbb{R}^+$ be a scaling function. Denote by $\overline{E}$ the closure of $E$ in $\mathbb{R}^+$ and define a subset $M=M_{E, \mu}$ of the set $\mathbb{N}$ by the following rule:

($\textit{i}_1$) An integer number $m \geq 2$ belongs $M$ if and only if
\begin{equation*}
\left[\mu(m+1), \mu(m-1) \right] \cap \overline{E} \neq \emptyset
\end{equation*}
where $\left[\mu(m+1), \mu(m-1) \right]= \{ x \in \mathbb{R}^+ :\mu(m+1) \leq x \leq \mu(m-1) \}$Ý;

($\textit{i}_2$) The number $1$ belongs to $M$ if and only if $[\mu(2), \infty) \cap \overline{E} \neq \emptyset$.

\begin{proposition}\label{p4}
The following conditions hold for every $E \subseteq \mathbb{R}^+$ and every scaling function $\mu:\mathbb{N} \to \mathbb{R}^+$,

($\textit{i}$) $M_{E, \mu}$ is empty if and only if $E \subseteq \{0\}$;

($\textit{ii}$) $M_{E, \mu}$ is finite if and only if $0 \notin acE$;

($\textit{iii}$) The equality $M_{E,\mu}= M_{\overline{E}, \mu}$ holds.

Moreover for every $\mu:\mathbb{N} \to \mathbb{R}^+$ and all $A, B \subseteq \mathbb{R}^+$ we have

($\textit{iv}$) If $A \subseteq B$, then the inclusion $M_{A, \mu} \subseteq M_{B, \mu}$ holds.

($\textit{v}$) The equality $M_{A \cup B, \mu}= M_{A, \mu} \cup M_{B, \mu}$ holds.
\end{proposition}
 \begin{proof}
 Property ($\textit{v}$) follows from the well-known equality
 \begin{equation*}
 \overline{A \cup B}= \overline{A} \cup \overline{B}.
 \end{equation*}
 Other properties can be derived directly from the definition of the function

 \begin{equation*}
 E \mapsto M_{E, \mu}.
 \end{equation*}
 \end{proof}

The next theorem shows that the porosity at infinity on $\mathbb{N}$ gives a relevant model for the porosity at $0$ on $\mathbb{R}^+$.

\begin{theorem}\label{th4}
Let $\mu:\mathbb{N} \to \mathbb{R}^+$ be a scaling function. Then the following statements are equivalent.

($\textit{i}$) The equality
\begin{equation}\label{eq43}
\textbf{P}(E) = \textbf{P}_{\mu}(M)
\end{equation}
holds for every $E \subseteq \mathbb{R}^+$ with $M= M_{E, \mu}$.

($\textit{ii}$) The scaling function $\mu$ satisfies the limit relation
\begin{equation}\label{eq44}
\lim_{n \to \infty} \frac{\mu(n+1)}{\mu(n)}=1.
\end{equation}
\end{theorem}
\begin{proof}
Let statement ($\textit{i}$) hold. If ($\textit{ii}$) is not valid, then using the decrease of $\mu$ we obtain
\begin{equation*}
\liminf_{n \to \infty} \frac{\mu(n+1)}{\mu(n)} < 1.
\end{equation*}
This inequality and Lemma \ref{l1} imply that
\begin{equation}\label{eq45}
\overline{p}_{\mu}(\mathbb{N}) >0.
\end{equation}
It is clear that the set $\mathbb{R}^+$ is nonporous, $\overline{p}(\mathbb{R}^+)=0$. From \eqref{eq43} it follows that
\begin{equation*}
\overline{p}(E)= \overline{p}_{\mu}(M)
\end{equation*}
holds for every $E \subseteq \mathbb{R}^+$. Since $M_{\mathbb{R}^+}, \mu= \mathbb{N}$ holds for every scaling function $\mu$, we obtain
\begin{equation*}
\overline{p}_{\mu}(\mathbb{N})=0,
\end{equation*}
contrary to inequality \eqref{eq45}.

Suppose now that statement ($\textit{ii}$) holds. Then by Proposition \ref{p5} we obtain the equality
\begin{equation*}
\textbf{P}_{\mu}(M)= \textbf{P}(\mu(M)).
\end{equation*}
Consequently equality \eqref{eq43} can be written in the form
\begin{equation}\label{eq46}
\textbf{P}(E) = \textbf{P}(\mu(M))
\end{equation}
where $M= M_{E, \mu}$. Equality \eqref{eq46} is trivial for $E$ satisfying the condition $0 \notin acE$. Indeed, by statement ($\textit{ii}$) of Proposition \ref{p4} the set $M_{E, \mu}$ is finite if $0 \notin acE$. Thus in this case $\textbf{P}(\mu(M))= \textbf{P}(E)= \{1\}$. Suppose $0 \in acE$. Corollary \ref{cor29} implies that \eqref{eq46} holds if $E \preceq M_{E, \mu}$ and $M_{E, \mu} \preceq E$. By the definition the statement $E \preceq M_E$ holds if for every $(e_n)_{n \in \mathbb{N}}$ with $e_n \in E \backslash \{0\}$, $n \in \mathbb{N}$ and $\lim_{n \to \infty} e_n =0$ there is a sequence $(x_n)_{n \in \mathbb{N}}$ with $x_n \in M \backslash \{0\}$ such that
\begin{equation}\label{eq47}
\lim_{n \to \infty} \frac{e_n}{x_n}=1.
\end{equation}
Using property ($\textit{i}_1$) from the definition of $M_{E, \mu}$ we see that for every sufficiently large $n \in \mathbb{N}$ there is $n(m) \in \mathbb{N}$ satisfying the statement
\begin{equation}\label{eq48}
e_n \in \left[ \mu(n(m)+1), \mu(n(m)-1) \right].
\end{equation}
Write $x_n = \mu(n(m))$. Since $\mu(n(m)+1) \leq \mu(n(m)) \leq \mu(n(m)-1)$, condition \eqref{eq48} gives us the estimations
\begin{equation*}
\frac{\mu(n(m)+1)}{\mu(n(m))} \leq \frac{e_n}{\mu(n(m))} \leq \frac{\mu(n(m)-1)}{\mu(n(m))}.
\end{equation*}
This double inequality and the limit relation
\begin{equation*}
\lim_{k \to \infty} \frac{\mu(k)}{\mu(k+1)} =1
\end{equation*}
imply \eqref{eq47}. The statement $E \preceq M_{E, \mu}$ follows. Reasoning similarly we obtain the statement $M_{E, \mu} \preceq E$.
\end{proof}

\section{The Lower Porosity At Infinity}

Define a family $\textbf{SSP}$ of sets $T \subseteq \mathbb{R}^+$ by the next rule. A set $T \subseteq \mathbb{R}^+$ belongs to $\textbf{SSP}$ if $ 0 \notin acT$ or there is a sequence $\{ (a_k, b_k) \}_{k \in \mathbb{N}}$ of open intervals $(a_k, b_k) \subseteq \mathbb{R}^+$ meeting the following conditions.

($i_1 $) The inequalities $a_k \geq b_{k+1} > a_{k+1} > 0$ hold for each $k \in \mathbb{N}$.

($i_2$) Every interval $(a_k, b_k)$ is a connected component of $ExtT$, i.e.
\begin{equation*}
(a_k, b_k) \cap T = \emptyset
\end{equation*}
but for every $(a, b) \supseteq (a_k, b_k)$ we have $\left( (a,b) \neq (a_k, b_k) \right) \Rightarrow \left( (a,b) \cap T \neq \emptyset \right)$.

($i_3$) The limit relations
\begin{equation}\label{eq33}
\lim_{k \to \infty} a_k =0, ~~ \lim_{k \to \infty} \frac{b_k - a_k}{b_k}=1 ~~\text{and}~ \lim_{k \to \infty} \frac{b_{k+1}}{a_k} =1
\end{equation}
hold.
\begin{remark}\label{r1}
The letters $\textbf{SSP}$ is merely an abbreviation for the words "super strongly porous". Every $T \in \textbf{SSP}$ is completely strongly porous in the sense of paper \cite{bd1}.
\end{remark}
\begin{theorem}\label{p3}
Let $\mu :\mathbb{N} \to \mathbb{R}$ be a scaling function. The equality
\begin{equation}\label{eq34}
\underline{p}_{\mu}(\mathbb{N}) =1
\end{equation}
holds if and only if $\mu(\mathbb{N}) \in \textbf{SSP}$.
\end{theorem}
\begin{proof}
Let $\mu(\mathbb{N}) \in \textbf{SSP}$. Then there is a sequence $\{ (a_k, b_k) \}_{k \in \mathbb{N}}$ satisfying the conditions from the definition of $\textbf{SSP}$ with $E=\mu(\mathbb{N})$.

In particular from ($i_1 $) and ($i_2 $) it follows that
\begin{equation*}
\mu(\mathbb{N}) \subseteq [ b_1, \infty) \cup^{\infty}_{k =1} [a_k, b_{k+1}].
\end{equation*}
Consequently there is $N_0 \in \mathbb{N}$ such that for every $n \geq N_0$ there is a unique $k=k(n)$ such that
\begin{equation}\label{eq35}
\mu(n) \in [b_{k+1}, a_k].
\end{equation}
Condition ($i_2$) from the definition of $\textbf{SSP}$ implies that $a_k, b_k \in \mu(\mathbb{N})$ for every $k \in \mathbb{N}$. Hence, for every $n \in \mathbb{N}$ with $n > N_0$ we have
\begin{equation}\label{eq36}
\lambda_{\mu}(\mathbb{N},n) \geq |b_{k+1}- a_{k+1}|
\end{equation}
where $\lambda_{\mu}(\mathbb{N},n)$ is defined as in \eqref{eq6}. Using inequality \eqref{eq36}, limit relations \eqref{eq33} and condition \eqref{eq35} we obtain
\begin{eqnarray*}
\underline{p}_{\mu}(\mathbb{N}) &=& \liminf_{n \to \infty} \frac{\lambda_{\mu}(\mathbb{N},n)}{\mu(n)} \\
&\geq& \liminf_{k \to \infty} \frac{b_{k+1}- a_{k+1}}{a_k} \\
&=& \liminf_{k \to \infty} \frac{b_{k+1}- a_{k+1}}{b_{k+1}} \frac{b_{k+1}}{a_k} \\
&=& \lim_{k \to \infty} \frac{b_{k+1}- a_{k+1}}{b_{k+1}} \lim_{k \to \infty} \frac{b_{k+1}}{a_k} =1.
\end{eqnarray*}
Consequently $\underline{p}_{\mu}(\mathbb{N}) \geq 1$. The inequality $\underline{p}_{\mu}(\mathbb{N}) \leq 1$ is trivial. Equality \eqref{eq34} follows.

To prove the implication
\begin{equation}\label{eq37}
\left( \underline{p}_{\mu}(\mathbb{N}) = 1 \right) \Rightarrow \left( \mu(\mathbb{N}) \in \textbf{SSP} \right)
\end{equation}
we will use the next fact. If $E \subseteq \mathbb{N}$ is infinite and $\underline{p}_{\mu}(\mathbb{N}) > \frac{1}{2}$, then there is $N_0 \in \mathbb{N}$ such that for every integer $n \geq N_0$ there are unique $n_1, n_2 \in \mathbb{N}$ satisfying the conditions $n \leq n_1 \leq n_2$, $(n_1, n_2) \cap E= \emptyset$ and
\begin{equation}\label{eq38}
\lambda_{\mu} (E,n) = |\mu(n_1)- \mu(n_2)|> \frac{1}{2} \mu(n)
\end{equation}
where $\lambda_{\mu} (E,n)$ is defined by \eqref{eq6}. Indeed, suppose the contrary and choose $m_1, m_2 \in \mathbb{N}$ such that \eqref{eq38} holds with $n_1 = m_1$ and $n_2 = m_2$. Consequently the inequality $\underline{p}_{\mu}(\mathbb{N}) > \frac{1}{2}$ implies that
\begin{equation}\label{eq39}
\frac{\mu(n_1)- \mu(n_2)}{\mu(n)} > \frac{1}{2} ~\text{and}~ \frac{\mu(m_1)-\mu(m_2)}{\mu(n)} > \frac{1}{2}.
\end{equation}
By our supposition we have
\begin{equation*}
(n_1, n_2) \neq (m_1, m_2).
\end{equation*}
Together with the equalities
\begin{equation*}
|\mu(n_1)- \mu(n_2)|= |\mu(m_1)- \mu(m_2)|= \lambda_{\mu}(E,n),
\end{equation*}
this implies that the intervals $(m_1, m_2)$ and $(n_1, n_2)$ are disjoint. Without loss of generality we may assume that
\begin{equation}\label{eq310}
n \leq n_1 < n_2 \leq m_1 < m_2.
\end{equation}
Now using \eqref{eq39} and \eqref{eq310} we obtain
\begin{eqnarray*}
\mu(n) &<& \left( \mu(n_1)- \mu(n_2) \right) + \left( \mu(m_1)- \mu(m_2) \right) \\
&=& \left( \mu(n_1)- \mu(m_2) \right) + \left( \mu(m_1)- \mu(n_2) \right) \\
&\leq&  \mu(n_1)- \mu(m_2) < \mu(n_1).
\end{eqnarray*}
Consequently the inequality $\mu(n) < \mu(n_1)$ holds, which contradicts $\mu$ being strictly decreasing. The desirable uniqueness follows.

Suppose now that $\underline{p}_{\mu}(\mathbb{N})=1$. Then there is $N_0 \in \mathbb{N}$ such that for every integer $n \geq N_0$ there is a unique interval $(n_1(n), n_2(n))$ such that \eqref{eq38} holds with $E = \mathbb{N}$.

Write $\tau(n)= n_2(n)$ for every $n \geq N_0$. Note that $n_1(n)= \tau(n)-1$ for every $n \geq N_0$ because $E= \mathbb{N}$ here. Define a sequence $\{ (a_k, b_k) \}_{k \in \mathbb{N}}$ by setting $a_1= \mu(\tau(N_0))$, $b_1= \mu(\tau(N_0)-1)$ and for $k \geq 2$, $a_k = \mu(\tau^k(N_0))$, $b_k = \mu(\tau^k(N_0)-1)$ where $\tau^k(N_0)= \tau(\tau^{k-1}(N_0))$ with $\tau^1(N_0)= \tau(N_0)$.

It remains to prove that conditions ($i_1 $)-($i_3 $) from the definition of $\textbf{SSP}$ are satisfied with $T= \mu(\mathbb{N})$ if $\{ (a_k, b_k) \}_{k \in \mathbb{N}}$ is defined as above. It is easy to see that equality \eqref{eq6} and the strict decreasing of $\mu$ imply ($i_1 $) and ($i_2 $).

Moreover the limit relation $\lim_{n \to \infty} \mu(n)=0$ yields $\lim_{k \to \infty}a_k=0$. Let us prove the equality
\begin{equation}\label{eq311}
\lim_{k \to \infty}\frac{b_k- a_k}{b_k}=1.
\end{equation}
It follows from the definitions that
\begin{equation*}
\frac{b_k- a_k}{b_k} = \frac{\mu(\tau^k(N_0)-1)- \mu(\tau^k(N_0))}{\mu(\tau^k(N_0)-1)} = \frac{\lambda_{\mu}(\mathbb{N},\tau^k(N_0)-1)}{\mu(\tau^k(N_0)-1)}.
\end{equation*}
Consequently we have
\begin{equation*}
\liminf_{k \to \infty} \frac{\lambda_{\mu}(\mathbb{N},\tau^k(N_0)-1)}{\mu(\tau^k(N_0)-1)} \in \textbf{P}_{\mu}(\mathbb{N}).
\end{equation*}
The equality $\underline{p}_{\mu}(\mathbb{N})=1$ implies that
\begin{equation*}
\textbf{P}_{\mu}(\mathbb{N})= \{ 1\}.
\end{equation*}
Hence the inequality
\begin{equation}\label{eq312}
\liminf_{k \to \infty} \frac{\lambda_{\mu}(\mathbb{N},\tau^k(N_0)-1)}{\mu(\tau^k(N_0)-1)} \geq 1
\end{equation}
holds. Moreover the inequality
\begin{equation}\label{eq313}
\limsup_{k \to \infty} \frac{\lambda_{\mu}(\mathbb{N},\tau^k(N_0)-1)}{\mu(\tau^k(N_0)-1)} \leq 1
\end{equation}
is evidently valid. Now \eqref{eq311} follows from \eqref{eq312} and \eqref{eq313}. Let us consider the last limit relation from \eqref{eq35},
\begin{equation}\label{eq314}
\lim_{k \to \infty} \frac{b_k}{a_{k-1}}=1.
\end{equation}
This equality can be written as
\begin{equation*}
\lim_{k \to \infty} \frac{\lambda_{\mu}(\mathbb{N},\tau^k(N_0)-1)}{\mu(\tau^k(N_0)-1)} =1.
\end{equation*}
Note that
\begin{equation*}
\lambda_{\mu}(\mathbb{N},n) = \mu(\tau^k(N_0)-1)- \mu(\tau^k(N_0))
\end{equation*}
for every integer $n \in [ \tau^{k-1}(N_0), \tau^k(N_0)-1 ]$. Similarly to \eqref{eq311}, we obtain the equality
\begin{equation*}
\lim_{k \to \infty} \frac{b_k - a_k}{a_{k-1}} =1.
\end{equation*}
The last equality and \eqref{eq311} directly imply \eqref{eq314}. The implication \eqref{eq37} follows.
\end{proof}

\begin{example}
Let $\mu:\mathbb{N} \to \mathbb{R}^+$ be a scaling function satisfying the limit relation
\begin{equation*}
\lim_{k \to \infty} \frac{\mu(k+1)}{\mu(k)} = 0.
\end{equation*}
It is easy to show that properties ($i_1 $)-($i_3 $) from the definitions of the family $\textbf{SSP}$ hold with $T= \mu(\mathbb{N})$ and $a_k = \mu(k+1)$, $b_k = \mu(k)$ for $k \in \mathbb{N}$. Consequently $\mu(\mathbb{N}) \in \textbf{SSP}$ and, by Theorem \ref{p3} the equality $\underline{p}_{\mu}(\mathbb{N})=1$ holds.
\end{example}

Theorem \ref{th2} claims that the equality
\begin{equation}\label{eq31}
\overline{p}(\mu(E))= \overline{p}_{\mu}(E)
\end{equation}
holds for every scaling function $\mu$ and every $E \subseteq \mathbb{N}$. Note that the equality arising from \eqref{eq31} by replacing of the upper porosity by lower porosity is generally not valid, as it follows from Theorem\ref{p2} and Proposition \ref{p6}.

\begin{proposition}\label{p6}
Let $T \in \textbf{SSP}$ and let $0 \in acT$. Then the equality
\begin{equation}\label{eq318}
\underline{p}(T) = \frac{1}{2}
\end{equation}
holds.
\end{proposition}
\begin{proof}
If $(a_n)$ is a sequence from the definition of $\textbf{SSP}$, then
\begin{equation*}
A= \{ a_1,..., a_k, a_{k+1},... \}
\end{equation*}
satisfies the conditions:
\begin{itemize}
  \item $a_k > a_{k+1} > 0$ for every $k \in \mathbb{N}$;
  \item $\lim_{k \to \infty} \frac{a_{k+1}}{a_k}=0$;
  \item $T \preceq A$ and $A \preceq T$
\end{itemize}
where the symbol $\preceq$ is understood in accordance with Definition \ref{dkck}. Using Corollary \ref{cor29} it is easy to see that \eqref{eq318} holds if and only if we have the equality
\begin{equation}\label{eq319}
\underline{p}(A)= \frac{1}{2}.
\end{equation}
Let us prove \eqref{eq319}. Let $(h_n)_{n \in \mathbb{N}}$ be a decreasing sequence of positive numbers such that
\begin{equation}\label{eq320}
\underline{p}(A)= \lim_{h \to \infty}\frac{\lambda(A, h_n)}{h}
\end{equation}
and $h_1 \leq a_2$. Then for every $n \in \mathbb{N}$ there is a unique $k= k(n) \in \mathbb{N}$ such that
\begin{equation*}
h_k \in [a_k, a_{k+1}].
\end{equation*}
Let us consider the function
\begin{equation*}
f(h)= \frac{\lambda(A,h)}{h}
\end{equation*}
on the interval $[a_k, a_{k-1}]$. The equality $\lim_{k \to \infty}\frac{t_{k+1}}{t_k}=0$ implies that
\begin{equation}\label{eq321}
|a_{k+1}- a_k| \leq |a_k- a_{k-1}|
\end{equation}
for sufficiently large $k$.

Using \eqref{eq321} it is easy to show that

\begin{equation*}
\lambda(A,h)=\left\{
\begin{array}{cc}
a_k-a_{k+1}, & h \in [a_k, a_k+(a_k- a_{k+1})] \\
h-a_k, &  \quad h \in [a_k+(a_k- a_{k+1}), a_{k-1}].~\ \ \ \ \ ~\ \ \ \ \ ~\
\end{array}%
\right.
\end{equation*}
Consequently we have

\begin{equation*}
f(n)=\left\{
\begin{array}{cc}
\frac{a_k- a_{k+1}}{h}, & h \in [a_k, 2a_k-a_{k+1}] \\
\frac{h-a_k}{h}, &  \quad h \in [2a_k- a_{k+1}, a_{k-1}].~\ \ \ \ \ ~\ \ \ \ \ ~\
\end{array}%
\right.
\end{equation*}

It implies that $f$ is decreasing on $[a_k, 2a_k - a_{k+1}]$ and increasing on $[2a_k- a_{k+1}, a_{k-1}]$. Hence we obtain
\begin{equation}\label{eq322}
\min_{h \in [a_k, a_{k-1}]}f(h)= f(2a_k - a_{k+1})= \frac{a_k - a_{k+1}}{2a_k - a_{k+1}}.
\end{equation}
Using \eqref{eq322} and the equality
\begin{equation*}
\lim_{k \to \infty} \frac{a_{k+1}}{a_k}=0
\end{equation*}
we can find that
\begin{equation*}
\underline{p}(A) \geq \lim_{k \to \infty} \frac{a_k - a_{k+1}}{2a_k - a_{k+1}}= \frac{1}{2}.
\end{equation*}
Since $\lim_{k \to \infty} \frac{a_k- a_{k+1}}{2a_k- a_{k+1}}= \lim_{k \to \infty}\frac{\lambda(A, 2a_k-a_{k+1})}{2a_k- a_{k+1}}$, the number $\frac{1}{2}$ is a porosity number of $A$. Hence the inequality $\frac{1}{2} \geq \underline{p}(A)$ holds. Equality \eqref{eq319} follows.
\end{proof}

\begin{corollary}\label{p2}
Let $E \subseteq \mathbb{R}^+$ and let $0$ be an accumulation point of $E$. Then the inequality
\begin{equation}\label{eq32}
\underline{p}(E) \leq \frac{1}{2}
\end{equation}
holds.
\end{corollary}
\begin{proof}
Since $0 \in acE$, there is a subset $T$ of $E$ such that $T \in \textbf{SSP}$ and $0 \in acT$. By Proposition \ref{p6} the equality $\underline{p}(T)= \frac{1}{2}$ holds. The inclusion $T \subseteq E$ implies that
\begin{equation*}
\lambda(E,h) \leq \lambda(T,h)
\end{equation*}
for every $h > 0$. Consequently we have the inequality
\begin{equation*}
\underline{p}(E) \leq \underline{p}(T).
\end{equation*}
This inequality and the equality $\underline{p}(T)= \frac{1}{2}$ give us \eqref{eq32}.
\end{proof}

\begin{definition}\label{d4}
A real valued sequence $(a_n)_{n \in \mathbb{N}}$ is eventually concave if the inequality
\begin{equation}\label{eq49}
\frac{a_{n-1}+ a_{n+1}}{2} \geq a_n
\end{equation}
holds for all sufficiently large $n$.
\end{definition}
\begin{proposition}\label{p7}
Let a scaling function $\mu$ be eventually concave. Then we have the equalities
\begin{equation}\label{eq410}
\underline{p}(\mu(\mathbb{N}))= \frac{\underline{p}_{\mu}(\mathbb{N})}{1+ \underline{p}_{\mu}(\mathbb{N})}
\end{equation}
and
\begin{equation}\label{eq411}
\underline{p}_{\mu}(\mathbb{N})= 1- \limsup_{n \to \infty}\frac{\mu(n+1)}{\mu(n)}.
\end{equation}
\end{proposition}
\begin{proof}
Let us prove \eqref{eq411} by analogy with the proof of the Lemma \ref{l1}. It is clear that
\begin{equation*}
1- \limsup_{n \to \infty} \frac{\mu(n+1)}{\mu(n)} = \liminf_{n \to \infty} \frac{\mu(n)- \mu(n+1)}{\mu(n)}.
\end{equation*}
Since $\mu$ is eventually concave, we have
\begin{equation*}
|\mu(n+1)-\mu(n)| \geq |\mu(n+2)- \mu(n+1)|
\end{equation*}
for all sufficiently large $n$. Consequently the equality
\begin{equation*}
\lambda_{\mu}(\mathbb{N},n)= \mu(n)- \mu(n+1)
\end{equation*}
holds for sufficiently large $n$. Hence
\begin{equation*}
\liminf_{n \to \infty} \frac{\mu(n)- \mu(n+1)}{\mu(n)}= \liminf_{n \to \infty} \frac{\lambda_{\mu}(\mathbb{N},n)}{\mu(n)}= \underline{p}_{\mu}(\mathbb{N}).
\end{equation*}
Equality \eqref{eq411} follows.

Let us prove \eqref{eq410}. As in the proof of Proposition \ref{p6} we can find that

\begin{equation}\label{eq412}
\underline{p}(\mu(\mathbb{N}))= \liminf_{k \to \infty} \frac{\mu(k)- \mu(k+1)}{2\mu(k)- \mu(k+1)}.
\end{equation}
(Note that inequality \eqref{eq321} holds if and only if the sequence $(a_n)_{n \in \mathbb{N}}$ from the proof of Proposition \ref{p6} is concave.) Using \eqref{eq412} and \eqref{eq411} we obtain
\begin{eqnarray*}
\underline{p}(\mu(\mathbb{N})) &=& 1- \limsup_{k \to \infty} \frac{\mu(k)}{2\mu(k)- \mu(k+1)} \\
&=& 1- \frac{1}{2- \limsup_{k \to \infty}\frac{\mu(k+1)}{\mu(k)}} = 1- \frac{1}{1- \underline{p}_{\mu}(\mathbb{N})},
\end{eqnarray*}
that implies \eqref{eq410}.
\end{proof}

The closing result of this section is the following infinitesimal characterization of super strongly porous sets.

\begin{theorem}\label{teo540}
Let $0 \in E \subseteq \mathbb{R}^+$. Then $E \in \textbf{SSP}$ if and only if the inequality
\begin{equation}\label{eq541}
card(\Omega^{E}_{0, \tilde{r}}) \leq  2
\end{equation}
holds for every pretangent space $\Omega^{E}_{0, \tilde{r}}$.
\end{theorem}

A relatively simple proof of this theorem can be obtained if we use the corresponding result for completely strongly porous at $0$ subsets of $\mathbb{R}^+$. In what follows the set of such subsets will be denoted by $\textbf{CSP}$. Several different characterizations of $\textbf{CSP}$-sets have been found in \cite{bd1}. In particular using Theorem 27, Theorem 42 and Definition 22 from \cite{bd1} we can give the definition of \textbf{CSP}-sets in the next form.

\begin{definition}\label{def541}
Let $E \subseteq \mathbb{R}^+$ and $0 \in acE$. Then $E$ is a $\textbf{CSP}$-set if there is a sequence $\tilde{L}=((a_n, b_n))_{n \in \mathbb{N}}$ satisfying the following conditions.

($\textit{i}_{1}$) Every interval $(a_n, b_n)$ is a connected component of $ExtE$.

($\textit{i}_{2}$) The inequalities $a_k \geq b_{k+1} > a_{k+1} > 0$ holds for each $n \in \mathbb{N}$.

($\textit{i}_{3}$) The upper limit
\begin{equation}\label{eq542}
M(\tilde{L}):= \limsup_{n \to \infty} \frac{a_n}{b_{n+1}}
\end{equation}
is finite.
\end{definition}
Now directly from the definitios of $\textbf{SSP}$-sets and $\textbf{CSP}$-sets we have the following.

\begin{lemma}\label{l543}
Let $E \subseteq \mathbb{R}^+$ and $0 \in acE$. Then $E$ is a $\textbf{SSP}$-set if and only if $E$ is a $\textbf{CSP}$-set and $M(\tilde{L})=1$ where $M(\tilde{L})$ is defined by \eqref{eq542}.
\end{lemma}

To formulate an infinitesimal characterization of $\textbf{CSP}$-sets we introduce the following two quantities

Let $\mathcal{F}= \{(X_i, d_i, p_i):i \in I \}$ be a family of metric spaces with marked points $p_i \in X_i$. Write
\begin{equation*}
D_i = \{ d_i(x,p_i):x \in X_i \}, ~i \in I.
\end{equation*}
Define
\begin{equation*}
\rho^{*}(X_i) = \sup_{t \in D_i}t ~\text{and}~R^{*}(\mathcal{F})= \sup_{i \in I}\rho{*}(X_i)
\end{equation*}
and, respectively,
\begin{equation*}
\rho_{*}(X_i)=\left\{\begin{array}{ccc}\inf\{t:t \in D_i \backslash \{0\}  \} & ~\text{if}~ & D_i \neq \{0\} \\+\infty & ~\text{if}~ & D_i=\{0\},\end{array}\right.
\end{equation*}
and
\begin{equation*}
R_{*}(\mathcal{F})=\inf \rho_{*}(X_i).
\end{equation*}
\begin{lemma}\label{l544}
Let $\mathcal{F}=\{(X_i, d_i, p_i):i \in I\}$ be a nonempty family of metric space with marked points $p_i$. Then the following statements are equivalent.

(\textit{i}) The equality $D_i=\{0,1\}$ holds for every $i \in I$.

(\textit{ii}) We have $R^{*}(\mathcal{F})= R_{*}(\mathcal{F})=1$.
\end{lemma}

\begin{proof}
The implication (\textit{i}) $\Rightarrow$ (\textit{ii}) is trivial. Suppose that (\textit{ii}) holds. If there is $i_0 \in I$ such that
\begin{equation*}
(0,1) \cap D_i \neq \emptyset,
\end{equation*}
then the strict inequality $R_{*}(\mathcal{F}) < 1$ holds contrary to (\textit{ii}). Hence $(0,1) \cap D_i =\emptyset$. Similarly we see that $(1, \infty)\cap D_i = \emptyset$ for every $i \in I$. The implication (\textit{ii}) $\Rightarrow$ (\textit{i}) follows.
\end{proof}

Let $0 \in E \subseteq \mathbb{R}^+$. Define the set $\mathbf{^1\Omega^{E}_{0}}$ of pretangent spaces $\Omega^{E}_{0, \tilde{r}}$ by the rule

\begin{equation*}
\left( \Omega^{E}_{0, \tilde{r}} \in \mathbf{^1\Omega^{E}_{0}} \right) \Leftrightarrow \left( 1 \in \overline{\Omega}^{E}_{0, \tilde{r}} \right).
\end{equation*}

The next lemma follows directly from Theorem 4.6 and formula (4.11) of \cite{bdk}.

\begin{lemma}\label{l545}
Let $0 \in E \subseteq \mathbb{R}^+$ and let $0 \in acE$. Then the following three conditions are equivalent.

(\textit{i}) The inequality
\begin{equation*}
R^{*}(\mathbf{^1\Omega^{E}_{0}}) < \infty
\end{equation*}
holds.

(\textit{ii}) The inequality
\begin{equation*}
R_{*}(\mathbf{^1\Omega^{E}_{0}}) > 0
\end{equation*}
holds.

(\textit{iii}) The set $E$ is a \textbf{CSP}-set.

Moreover if $E$ is a \textbf{CSP}-set, then
\begin{equation*}
R^{*}(\mathbf{^1\Omega^{E}_{0}})= M(\tilde{L}) ~\text{and}~R_{*}(\mathbf{^1\Omega^{E}_{0}})= \frac{1}{M(\tilde{L})}
\end{equation*}
where $M(\tilde{L})$ is defined by \eqref{eq542}.
\end{lemma}

Now we are ready to prove the theorem formulated above.

\textit{Proof of Theorem \ref{teo540}}  By Lemma \ref{l543} we have $E \in \textbf{SSP}$ if and only if
\begin{equation}\label{eq545}
E \in \textbf{CSP} ~\text{and}~M(\tilde{L})=1.
\end{equation}
Suppose that inequality \eqref{eq541} holds. We must show that conditions \eqref{eq545}  are satisfied. Let $\overline{\Omega}^{E}_{0, \tilde{r}} \in \mathbf{^1\Omega^{E}_{0}}$. Then inequality \eqref{eq541} implies the equality
\begin{equation*}
\overline{\Omega}^{E}_{0, \tilde{r}} = \{0,1\}.
\end{equation*}
It follows from the last equality that
\begin{equation*}
R^{*}(\mathbf{^1\Omega^{E}_{0}})= R_{*}(\mathbf{^1\Omega^{E}_{0}})=1.
\end{equation*}
Hence, by Lemma \ref{l544}, conditions \eqref{eq545} hold. Similarly if conditions \eqref{eq545} hold, then inequality \eqref{eq541} follows from Lemma \ref{l544} and Lemma \ref{l545}. \qquad\qquad\qquad$\Box$

\medskip

Define a function $F:E\times E \to \mathbb{R}^+$ as
\begin{equation*}
F(x,y)=\left\{\begin{array}{ccc}\frac{|x-y|(x \wedge y)}{(x \vee y)^2} & \text{if} & (x,y) \neq (0,0) \\0 & \text{if} & (x,y)=(0,0)\end{array}\right.
\end{equation*}
where $x \vee y= \max\{x,y\}$ and $x \wedge y= \min\{x,y\}$. Using Theorem \ref{teo540} which was proved above and Theorem 2.2 from \cite{dovgoshey2} we obtain the following.

\begin{corollary}\label{son}
Let $0 \in E \subseteq \mathbb{R}$. Then $E$ is a \textbf{SSP}-set if and only if
\begin{equation*}
\lim_{x,y \to 0} F(x,y)=0.
\end{equation*}
\end{corollary}
\textbf{Acknowledgements}

The research of the second author was supported by a grant received from TUB\.{I}TAK within 2221-Fellowship Programme for Visiting Scientists and Scientists on Subbatical Leave.

\end{document}